\documentclass[11pt]{amsart}

\usepackage{amsfonts}
\usepackage{amssymb}
\usepackage{graphicx}
\usepackage{amsmath}
\usepackage{latexsym}
\usepackage{amscd}
\usepackage{xypic}
\usepackage[all,cmtip]{xy}
\usepackage{mathrsfs}%scrfont}
\usepackage{enumitem} %http://ctan.org/pkg/enumitem
\usepackage{braket}
\usepackage[margin=1.4in]{geometry}
\usepackage{esint}
\usepackage{dsfont}
\usepackage{pgf}
\usepackage{tikz}
\usetikzlibrary{decorations.pathreplacing}

\usepackage{hyperref}

\allowdisplaybreaks 

\numberwithin{equation}{section}

\newtheorem{prop}{Proposition}[section]
\newtheorem{theo}[prop]{Theorem}
\newtheorem{lemm}[prop]{Lemma}
\newtheorem{coro}[prop]{Corollary}

\newtheorem*{claim*}{Claim}

\theoremstyle{definition}

\newtheorem{rema}[prop]{Remark}
\newtheorem{exam}[prop]{Example}

\newtheorem{defi}[prop]{Definition}

\newcommand{\NN}{\mathbb{N}}

\newcommand{\QQ}{\mathbb{Q}}
\newcommand{\RR}{\mathbb{R}}
\renewcommand{\SS}{\mathbb{S}}
\newcommand{\TT}{\mathbb{T}}
\newcommand{\ZZ}{\mathbb{Z}}

\newcommand{\cB}{\mathcal B}

\renewcommand{\cL}{\mathcal L}

\newcommand{\cS}{\mathcal S}

\newcommand{\cU}{\mathcal U}

\newcommand{\bangle}[1]{\left\langle #1 \right\rangle}

\DeclareMathOperator{\area}{area}
\DeclareMathOperator{\Ric}{Ric}

\define{\sff}{{\rm II}}
\define{\tfsff}{\accentset{\circ}{\sff}}

\let\oldmarginpar\marginpar
\renewcommand\marginpar[1]{\-\oldmarginpar[\raggedleft\footnotesize #1]%
{\raggedright\footnotesize #1}}

\DeclareMathOperator{\Index}{index}
\DeclareMathOperator{\genus}{genus}

\DeclareMathOperator{\vol}{vol}

\usepackage{graphicx}

\allowdisplaybreaks

\title[Minimal hypersurfaces with bounded index]{Minimal hypersurfaces with bounded index}

\author{Otis Chodosh}
\address{Department of Pure Mathematics and Mathematical Statistics, Cambridge University}
\curraddr{Department of Mathematics, Princeton University}
\email{ochodosh@math.princeton.edu}

\author{Daniel Ketover}
\address{Department of Mathematics, Princeton University}
\email{dketover@math.princeton.edu}

\author{Davi Maximo}
\address{Department of Mathematics, Stanford University}
\curraddr{Department of Mathematics, University of Pennsylvania} 
\email{dmaxim@math.upenn.edu}

\date{\today}

\begin{document}
	
\begin{abstract}
We prove a structural theorem that provides a precise local picture of how a sequence of closed embedded minimal hypersurfaces with uniformly bounded index (and volume if the ambient dimension is greater than three) in a Riemannian manifold $(M^{n},g)$, $3\leq n\leq 7$, can degenerate.  Loosely speaking, our results show that embedded minimal hypersurfaces with bounded index behave qualitatively like embedded stable minimal hypersurfaces, up to controlled errors. Several  compactness/finiteness theorems follow from our local picture.  
\end{abstract}

\maketitle

%\textbf{check appendix C. reference to xin's paper. }

\section{Introduction}
Minimal hypersurfaces are critical points of the volume functional, and as such it is natural to study their existence and behavior from a variational point of view. A key invariant related to this point of view is the (Morse) index of such an object; and the assumption of bounded index (rather than genus or total curvature as in the classical theory) that we are concerned with in this paper is very natural.  Many minimal surfaces with bounded index are expected to arise from the variational min-max theory of Almgren--Pitts \cite{Pit76}. For instance, Marques--Neves \cite{MaNe13} have introduced non-trivial $k$-parameter sweep-outs in arbitrary three-manifolds (for any $k$), and the Morse index of the corresponding minimal surface is expected (generically) to be $k$. Moreover, Colding--Gabai \cite{ColdingGabai} have recently studied sequences of index one minimal surfaces and as they relate to the problem of classifying Heegaard splittings of three-manifolds.

In this work, we provide a precise local picture of how a sequence of embedded minimal hypersurfaces with uniformly bounded index (and volume if the ambient dimension is greater than three) in a Riemannian manifold $(M^{n},g)$, $3\leq n\leq 7$, can degenerate. We may roughly describe it as follows. For the sake of exposition, we assume here that $n=3$, the surfaces are all two-sided and have uniformly bounded area, $i.e.$, consider a sequence of embedded two-sided minimal surfaces $\Sigma_j$ in a closed Riemannian manifold $(M^3,g)$ so that $\Index(\Sigma_{j})\leq I$ and $\area_{g}(\Sigma_{j})\leq \Lambda$. Given these assumptions, our results imply that $\genus(\Sigma_{j})$ is uniformly bounded:
\begin{enumerate}[itemsep=5pt, topsep=5pt]
\item A blow-up argument allows us to extend Schoen's curvature estimates \cite{Sch83} to the case of bounded index (cf.\ Corollary \ref{coro-curv-bds}), to see that the curvature of $\Sigma_{j}$ is bounded away from at most $I$ points, where the index may be concentrating. Blowing up around these points at the scale of curvature, we produce a smooth non-flat embedded minimal surface in $\RR^{3}$ with index at most $I$. 
\item By \cite{Ros:oneSided} (see also \cite{ChMa14}), this limiting surface has genus bounded linearly above in terms of $I$. This can be seen as a kind of ``lower semi-continuity of topology,'' since any genus that is seen in the blow-up limit certainly contributes to the genus of $\Sigma_{j}$ for $j$ sufficiently large.
\item Furthermore, after passing to a subsequence, $\Sigma_{j}$ converges to a closed minimal surface in $(M,g)$ smoothly away from the points of index concentration. Such a surface has bounded genus.
\end{enumerate}
These facts, by themselves, are not sufficient to conclude that the genus of $\Sigma_{j}$ is uniformly bounded. The reason for this is that it is a priori possible that some genus is lost to the intermediate scales, and thus does not appear in the blow-up limits, or at the original scale. This can be illustrated by an analogy with the bubbling phenomenon for harmonic maps; a sequence of harmonic maps may degenerate to form a bubble tree and the key point in proving that the energy of the bubbles is the limit of the energy of the original sequence of maps, is to show that no energy is lost in the neck regions joining the bubbles. Hence, to prove genus bounds, we must also show:
\begin{enumerate}[resume, itemsep=5pt, topsep=5pt]
\item No topology is lost in the intermediate scales. The key to this is a scale-breaking Morse theoretic argument (cf.\ Lemma \ref{lemm:ann-decomp}), which allows us to show that the intermediate regions are topologically simple (i.e., planar domains). The key geometric input for this argument is the fact that the curvature is sufficiently small in a scale invariant sense in the intermediate region, a fact related to the half-space theorem for complete properly embedded minimal surfaces in $\RR^{3}$. 
\end{enumerate}

This scale-breaking analysis of the intermediate regions forms the technical heart of our work. In ambient manifolds $(M^{n},g)$ with $4\leq n\leq 7$, a similar argument works given appropriate modifications. Moreover, for $n=3$, we show that uniform area bounds are \emph{unnecessary} for understanding the local picture of degeneration. In fact, we will be able to use the above argument to prove that index bounds imply area bounds (in addition to genus bounds) in $3$-manifolds of positive scalar curvature in Theorem \ref{theo:area-genus-bd-PSC}.

\subsection{Applications of the local picture of degeneration}

Thanks to our understanding of how embedded hypersurfaces with uniformly bounded index (and volume) can degenerate, we can prove various results along the lines of the general principle that (when the ambient dimension satisfies $3\leq n\leq 7$) ``embedded minimal hypersurfaces with uniformly bounded index behave qualitatively like embedded stable minimal hypersurfaces.'' We now discuss several results along these lines. 

\subsubsection{Finitely many diffeomorphism types}  An easy application of curvature estimates for stable minimal hypersurfaces shows that for a closed Riemannian manifold $(M^{n},g)$, $3\leq n\leq 7$, there can be at most $N = N(M,g,\Lambda)$ distinct diffeomorphism types in the set of stable embedded minimal hypersurfaces with $\vol_{g}(\Sigma)\leq\Lambda$. To see this, suppose that $\Sigma_{j}$ is an infinite sequence of pairwise non-diffeomorphic embedded stable minimal hypersurfaces. Using the curvature estimates\footnote{Note that the works \cite{SSY,Schoen-Simon:1981} only explicitly consider (embedded) two-sided stable hypersurfaces. However, using the fact that a properly embedded hypersurface in Euclidean space is two-sided, we can extend the curvature estimates to the one-sided case as well; see the proof of Lemma \ref{lemm:curv-est-nD} as well as Lemma \ref{lem:two-sided-small-balls}.} established in \cite{Sch83,Ros:oneSided,SSY,Schoen-Simon:1981} we see that if $\Sigma_{j}$ is a sequence of embedded\footnote{Strictly speaking, embeddedness is not needed for this result in dimension $n=3$. Note that it will be essential elsewhere in our work even when $n=3$.} stable minimal hypersurfaces, then there is $C>0$ so that $|\sff_{\Sigma_{j}}|(x) \leq C$ for all $x \in \Sigma_{j}$. Because $\vol_{g}(\Sigma_{j})\leq \Lambda$, by passing to a subsequence we may find $\Sigma_{\infty}$ so that $\Sigma_{j}$ converges locally smoothly to $\Sigma_{\infty}$ with finite multiplicity. Thus, for $j$ sufficiently large we may construct a smooth covering map $\Sigma_{j} \to \Sigma_{\infty}$ with a uniformly bounded number of sheets. Because the $\Sigma_{j}$ are assumed to be non-diffeomorphic, this easily yields a contradiction.\footnote{This follows from the fact that there are only finitely many different $k$-sheeted covers of a given compact manifold $\Sigma_{\infty}$. This can be proven by topological considerations, or as pointed out to us by the referee, this is an immediate consequence of a theorem of Hall (cf.\ \cite[Theorem 21.4]{Bogopolski}) which says that the number of subgroups of finite index $k$ of a finitely generated group is finite. Note that here we are implicitly using the fact that the smooth structure on the base and the topological data of the covering map uniquely determine the smooth structure of the cover (because the covering map is smooth).} 

To try to extend this proof to the case of uniformly bounded index, we must contend with the possibility that the hypersurfaces have diverging curvature. Using our local picture of degeneration, we can deal with this possibility and show the following finiteness result. 
\begin{theo}\label{theo:fin-top-type-Mn}
Fix $(M^{n},g)$ a closed Riemannian manifold, where $3\leq n\leq 7$. Then there can be at most $N = N(M,g,\Lambda,I)$ distinct diffeomorphism types in the set of embedded minimal hypersurfaces $\Sigma\subset (M,g)$ with $\Index(\Sigma) \leq I$ and $\vol_{g}(\Sigma) \leq\Lambda$.

In particular, for a closed three-manifold $(M^{3},g)$, there is $r_{0} = r_{0}(M,g,\Lambda,I)$ so that any embedded minimal surface $\Sigma$ in $(M^{3},g)$ with $\Index(\Sigma) \leq I$ and $\area_{g}(\Sigma)\leq \Lambda$ has $\genus(\Sigma)\leq r_{0}$. 
\end{theo}

In a related direction, we can partially extend Ros's bounds \cite[Theorem 17]{Ros:oneSided} (see also \cite{ChMa14}) to higher dimensions as follows. 

\begin{theo}\label{theo:fin-top-type-Rn}
For $4\leq n \leq 7$, there is $N=N(n,I,\Lambda) \in \NN$ so that there are at most $N$ mutually non-diffeomorphic complete embedded minimal hypersurfaces $\Sigma^{n-1}\subset \RR^{n}$ with $\Index(\Sigma) \leq I$ and $\vol(\Sigma\cap B_{R}(0)) \leq \Lambda R^{n-1}$ for all $R>0$. 
\end{theo}

It would be interesting to understand how $N$ depends on $I$ and $\Lambda$.\footnote{Added in proof: recently, Li has shown that for $n=4$, the $N$ in Theorem \ref{theo:fin-top-type-Rn} can be bounded independently of the area growth bound ($\Lambda < \infty$) \cite{Li:index-high-dim}.}

\subsubsection{Three-dimensional results} A well known compactness result for minimal surfaces in a fixed three-manifold is due to Choi--Schoen \cite{ChSc85} who showed that for any sequence of minimal surfaces with bounded genus and area there is a subsequence converging to a smooth minimal surface (possibly with multiplicity).  The convergence is moreover smooth away from finitely many points where curvature is concentrating. 

This result has several important manifestations; for example, in a hyperbolic manifold, a genus bound for a minimal surface already implies an area bound by the Gauss equation and the Gauss-Bonnet formula. Moreover, in a three-manifold with positive Ricci curvature, the area of an embedded minimal surface can be bounded above in terms of its genus, by work of Choi--Wang \cite{ChoiWang} and Yang--Yau \cite{YangYau}. This bound and the non-existence of two-sided stable minimal hypersurfaces with ambient positive Ricci curvature, shows that the Choi--Schoen compactness implies that the set of closed, embedded minimal surfaces with fixed genus in a three-manifold of positive Ricci curvature is compact in the smooth topology. In a general Riemannian three-manifolds, on the other hand, it is no longer possible to bound the area nor the index of an embedded minimal surface by the genus, even if one assumes positive scalar curvature, as it can be seen in examples constructed by Colding--De Lellis \cite{CD}. 

However, we are able to show that in three-manifolds with positive scalar curvature, uniform index bounds do imply uniform area and genus bounds. This indicates that index bounds are not only very natural from the variational point of view, but they actually yield more control on the minimal surface than genus bounds. That such a result should hold follows again from our general principle that because this holds for embedded stable minimal surfaces, it should hold for an embedded minimal surface with bounded index. The corresponding result in the case of stable surfaces is a consequence of the fact that by work of Fischer-Colbrie--Schoen \cite{Fischer-Colbrie-Schoen} and Schoen--Yau \cite{ScYa83}, two-sided stable minimal surfaces in ambient manifolds with positive scalar curvature are $\SS^{2}$, along with a geometric compactness argument based on the fact that $\SS^{2}$ is simply connected.   
\begin{theo}\label{theo:area-genus-bd-PSC}
Suppose that $(M^{3},g)$ is a closed three-manifold with positive scalar curvature. For $I \in \NN$, there is $A_{0}=A_{0}(M,g,I)<\infty$ and $r_{0}=r_{0}(M,g,I)$ so that if $\Sigma\subset (M,g)$ is a connected, closed, embedded minimal surface with $\Index(\Sigma)\leq I$, then $\area_{g}(\Sigma)\leq A_{0}$ and $\genus(\Sigma)\leq r_{0}$. 
\end{theo}

\begin{rema}
Unfortunately, without any extra assumption, even the space of embedded stable minimal surfaces fails to be compact in general three-manifolds due to the failure of uniform area bounds. We discuss several examples in \S \ref{subsect:counterexamples} below. In the converse direction, Ejiri--Micallef have shown \cite{EjiriMicallef} that for immersed minimal surfaces in a general three-manifold, uniform bounds on their area and genus imply uniform bounds on their index. 
\end{rema}

In a more technical direction, we remark that as a byproduct of the proof of Theorem \ref{theo:neck}, we obtain:
\begin{theo}\label{theo:removsing}
Suppose that $\Sigma_{j}$ is a sequence of embedded minimal surfaces in a three-manifold $(M,g)$ with uniformly bounded index, i.e., $\Index(\Sigma_{j})\leq I$ for some $I \in \NN$. Then, after passing to a subsequence, $\Sigma_{j}$ converges to a lamination $\cL$ away from at most $I$ singular points. The lamination can be extended across these points. 
\end{theo}

\begin{rema}\label{rema:MPR}
We note that the fact that the limit lamination $\cL$ has removable singularities can be seen as a consequence of deep work by Meeks--Perez--Ros \cite{MePeRo13}, combined with our curvature estimates for $\Sigma_{j}$, which after passing to the limit, imply\footnote{Alternatively, \cite[(3.1)]{LiZhou} establishes similar curvature estimates for the limit lamination; the exact form of estimates we establish here (before passing to the limit) are crucial for our proof of Theorem \ref{theo:neck} in several other places.}  that $|\sff_{\cL}|(x)d_{g}(x,\cB_{\infty})\leq C$ for $x \in \cL$.  Our proof of Theorem \ref{theo:removsing} however does not rely on the removable singularity results in \cite{MePeRo13}, and thus provides a self-contained proof that the limit lamination of a sequence of embedded closed minimal surfaces with bounded index has removable singularities. 
\end{rema}

\begin{rema}
We also remark that Theorem \ref{theo:neck} and Corollary \ref{coro:snip} provide an alternative approach to a recent result by Colding--Gabai, \cite[Theorem 2.2]{ColdingGabai} (cf.\ Remark \ref{rema:colding-gabai}). We will not reproduce the full statement here, but only note that it loosely says that a degenerating sequence of index-one embedded minimal surfaces will look like a small catenoid connected by large annular regions to the rest of the surface. 
\end{rema}

We remark that as consequence of the Theorems \ref{theo:fin-top-type-Mn} and \ref{theo:area-genus-bd-PSC}, we may easily deduce several compactness results. By a theorem of Colding--Minicozzi \cite{ColdingMinicozzi:no-area-bds}, the set of closed embedded minimal surfaces with uniformly bounded area and genus is finite in $(M^3,g)$, as long as  $g$ is ``bumpy'' in the sense of White \cite{White:bumpy2}, i.e. $g$ has the property that there are no immersed minimal submanifolds with non-zero Jacobi fields.\footnote{It seems to us that in general, the notion of ``bumpy'' from \cite{White:bumpy2}, rather than the notion from \cite{White:bumpy} is necessary to deal with the possibility of a one-sided limit in the proof of \cite{ColdingMinicozzi:no-area-bds}. See also \cite[Remark 3.1]{Carlotto:generic}.} Such metrics are ``generic'' by the main result in \cite{White:bumpy2}. Thus, we have:
\begin{coro}\label{coro:PSC-bumpy}
Suppose that $(M^{3},g)$ is a closed three-manifold with a bumpy metric of positive scalar curvature. For $I \in \NN$, there are only finitely many closed, connected, embedded minimal surfaces $\Sigma$ with $\Index(\Sigma)\leq I$. 
\end{coro}
\begin{rema}
A slightly different version of this corollary has recently been independently obtained by Carlotto \cite{Carlotto:generic}, assuming the ambient positive scalar curvature metric is bumpy in the sense of \cite{White:bumpy} (i.e., there is no embedded---as opposed to immersed---minimal submanifold with a non-zero Jacobi field) but with the additional assumption that $(M^{3},g)$ contains no embedded minimal $\RR P^{2}$.
\end{rema}

 Combining Theorem \ref{theo:area-genus-bd-PSC} with the work of Choi--Schoen \cite{ChSc85}, we also have:
\begin{coro}
Suppose that $(M^{3},g)$ is a closed three-manifold with positive Ricci curvature. Then, for $I \in \NN$, the set of closed, connected, embedded minimal surfaces $\Sigma$ with $\Index(\Sigma) \leq I$ is compact in the smooth topology. 
\end{coro}

In particular, combined with the recent work of Marques--Neves \cite{MaNe13} we obtain
\begin{coro}
Suppose that $(M^{3},g)$ is a closed three-manifold with a bumpy metric of strictly positive Ricci curvature. Then, there exists a sequence of closed, embedded minimal surfaces $\Sigma_{j}$ with $\Index(\Sigma_{j})\to\infty$. 
\end{coro}
\begin{rema}These last two corollaries have been recently proven by Li--Zhou \cite{LiZhou} by somewhat different arguments.
\end{rema}

\subsection{Counterexamples}\label{subsect:counterexamples}
Several examples show that the set of closed, embedded, stable minimal surfaces can fail to be compact, even if the metric is bumpy and even if we restrict only to the set of such surfaces with a fixed genus. The examples below show that the hypothesis in the applications discussed above cannot be significantly weakened. 

\begin{exam}
The simplest example of non-compactness occurs in the square three-torus $\TT^3=\RR^{3}/\ZZ^{3}$, equipped with the flat metric, as seen by choosing positive rational numbers $\theta_{k} \in \QQ$ converging to an irrational number $\theta_{\infty} \in \RR\setminus \QQ$. Letting $\gamma_{k}$ denote the simple closed geodesic in the two-torus $\TT^2=\RR^{2}/\ZZ^{2}$ with slope $\theta_{k}$, it is easy to see that $\Sigma_{k} : = \gamma_{k}\times \SS^{1}$ is an embedded stable minimal surface with $\area_{g}(\Sigma_{k})\to \infty$. Note that the surfaces $\Sigma_{k}$ limit to the lamination of $\TT^3$ by a single plane $\gamma_{\infty}\times \SS^{1}$ (where $\gamma_{\infty}$ is the non-closed geodesic with slope $\theta_{\infty}$). Of course, the flat metric on $\TT^{3}$ is manifestly not bumpy, but it is relatively easy to see that for an arbitrary metric on $\TT^{3}$, we can minimize (by \cite{ScYa79}) the $g$-area of immersions homotopic to the embedding of $\Sigma_{k}$ into $\TT^{3}$ and then argue (using \cite{BrWo69,FHS83} and fundamental group considerations) that this yields a sequence of embedded, stable, minimal tori in $(\TT^{3},g)$ with unbounded area. %Below, we will discuss more details in a more general setting. 
\end{exam}

One might hope that the torus $\TT^{3}$ is somehow special in the previous example. However, the following results show that for \emph{any} closed three-manifold, it is not possible to use bumpiness (or more generally, any ``generic'' property which is satisfied by a $C^{2}$-dense set of metrics) to prove area bounds for embedded stable minimal surfaces (even assuming fixed genus). 
\begin{exam}\label{exam:twist-tori}
Fix a Riemannian three-manifold $M$. Work of Colding--Minicozzi \cite{ColdingMinicozzi:no-area-bds} shows that there is a $C^{2}$-open set of metrics $g$ on $M$ so that there is a sequence of embedded stable minimal tori $\Sigma_{j}$ with $\area_{g}(\Sigma_{j})\to\infty$. Indeed, for any domain in $M$ of the form $\Omega = S \times \SS^{1}$ where $S$ is a disk with three holes removed (i.e., $\Omega$ is a solid torus with three holes removed, which obviously exists in any coordinate chart), they show that if $\Omega$ has strictly mean convex boundary with respect to a metric $g$, then there exists such a sequence of tori in $\Omega\subset (M,g)$. Their construction relies on an idea of ``looping'' tori around the holes; see \cite[Figure 3.2.3]{Kra09} for a nice illustration. These examples were subsequently extended by Dean \cite{Dea03} and Kramer \cite{Kra09} to give examples of sequences of embedded, stable, minimal surfaces with unbounded area, with any fixed genus. 
\end{exam}

\begin{exam}
An even more extreme example similar to Example \ref{exam:twist-tori} but with a more complicated looping scheme was given by Colding--Hingston \cite{CoHi06}, who in a $C^{2}$-open set of metrics on any three-manifold, construct a sequence of stable tori with unbounded area whose limit lamination has surprising behavior.
\end{exam}

In a more topological vein, we have the following examples of embedded minimal surfaces with bounded index (in fact stable) but unbounded genus (and hence area). 
\begin{exam}
For $\Sigma_{r}$ the closed oriented surface of genus $r>1$, Jaco proved \cite{Jac70} that $M_{r}:=\Sigma_{r}\times \SS^{1}$ admits a sequence of incompressible surfaces with unbounded genus. For any Riemannian metric $g$ on $M_{r}$, we may minimize area using \cite{ScYa79} and see that the resulting stable minimal surface is embedded (after passing to a one-sided quotient, if necessary) by \cite{FHS83}. It is clear that these minimal surfaces must have unbounded genus. 
\end{exam}

\subsection{Precise statement of degeneration and surgery results in $3$-dimensions}
We now state our main results in three-dimensions. 
\begin{theo}[Local picture of degeneration] \label{theo:neck}
There are functions $m(I)$ and $r(I)$ with the following property. Fix a closed three-manifold $(M^{3},g)$ and a natural number $I \in \NN$. Then, if $\Sigma_{j}\subset (M,g)$ is a sequence of closed embedded minimal surfaces with
\[
\Index(\Sigma_{j})\leq I,
\]
then after passing to a subsequence, there is $C>0$ and a finite set of points $\cB_{j}\subset \Sigma_{j}$ with cardinality $|\cB_{j}|\leq I$ so that the curvature of $\Sigma_{j}$ is uniformly bounded away from the set $\cB_{j}$, i.e.,
\[
|\sff_{\Sigma_{j}}|(x) \min\{1,d_{g}(x,\cB_{j})\}\leq C,
\]
but not at $\cB_{j}$, i.e.,
\[
\liminf_{j\to\infty} \min_{p\in\cB_{j}}|\sff_{\Sigma_{j}}|(p) = \infty.
\]

Passing to a further subsequence, the points $\cB_{j}$ converge to a set of points $\cB_{\infty}$ and the surfaces $\Sigma_{j}$ converge locally smoothly, away from $\cB_{\infty}$, to some lamination $\cL \subset M \setminus \cB_{\infty}$. The lamination has removable singularities, i.e., there is a smooth lamination $\widetilde\cL\subset M$ so that $\cL = \widetilde\cL\setminus\cB_{\infty}$. Moreover, there exists $\varepsilon_{0}>0$ smaller than the injectivity radius of $(M,g)$ so that $\cB_{\infty}$ is $4\varepsilon_{0}$-separated and for any $\varepsilon \in (0,\varepsilon_{0}]$, taking $j$ sufficiently large guarantees that
\begin{enumerate}[itemsep=5pt, topsep=5pt]
\item Writing $\Sigma_{j}'$ for the components of $\Sigma_{j}\cap B_{2\varepsilon}(\cB_{\infty})$ containing at least one point from $\cB_{j}$, no component of $\Sigma_{j}'$ is a topological disk, so we call $\Sigma_{j}'$ the ``neck components.'' They have the following additional properties:
\begin{enumerate}[itemsep=5pt, topsep=5pt]
\item The surface $\Sigma_{j}'$ intersects $\partial B_{\varepsilon}(\cB_{\infty})$ transversely in at most $m(I)$ simple closed curves.
\item Each component of $\Sigma_{j}'$ is unstable. 
\item The genus\footnote{See Definition \ref{defi:genus-bdry}.} of $\Sigma_{j}'$ is bounded above by $r(I)$.
\item The area of $\Sigma_{j}'$ is uniformly bounded, i.e.,
\[
\limsup_{j\to\infty} \area_{g}(\Sigma_{j}') \leq 2\pi m(I)\varepsilon^{2}(1 + o(\varepsilon))
\]
as $\varepsilon\to 0$.
\end{enumerate}
\item Writing $\Sigma_{j}''$ for the components of $\Sigma_{j}\cap B_{2\varepsilon}(\cB_{\infty})$ that do not contain any points in $\cB_{j}$, each component of $\Sigma_{j}''$ is a topological disk, so we call $\Sigma_{j}''$ the ``disk components.'' Moreover, we have the following additional properties 
\begin{enumerate}[itemsep=5pt, topsep=5pt]
\item The curvature of $\Sigma_{j}''$ is uniformly bounded, i.e.,
\[
\limsup_{j\to\infty}\sup_{x\in\Sigma_{j}''}|\sff_{\Sigma_{j}}|(x) < \infty.
\]
\item Each component of $\Sigma_{j}''$ has area uniformly bounded above by $2\pi \varepsilon^{2}(1+o(\varepsilon))$.
\end{enumerate}
\end{enumerate}
%\item The components of $\Sigma_{j}\cap B_{2\varepsilon}(\cB_{\infty})$ that intersect $\cB_{\varepsilon/2}(\cB_{\infty})$ intersect $\partial B_{\varepsilon}(\cB_{\infty})$ transversely. 
\end{theo}
As is clear from the proof, it would be possible to give explicit bounds for $m(I)$ and $r(I)$, if one desired. 

\begin{rema}\label{rema:colding-gabai}
As remarked above, a similar description in the special case of index one surfaces was recently obtained by Colding--Gabai \cite[Theorem 2.2]{ColdingGabai}. However, our proof differs from theirs (even in the index one case) in how we transfer topological information between scales. Additionally, the higher index case introduces serious technical difficulties, due to the possibility of simultaneous concentration at multiple scales. 
\end{rema}

A key application of Theorem \ref{theo:neck} is a prescription for performing ``surgery'' on a sequence of bounded index minimal surfaces so that their curvature remains bounded, while only changing the topology and geometry in a controllable way. 

\begin{coro}[Controlled surgery]\label{coro:snip}
There exist functions $\tilde r(I)$ and $\tilde m(I)$ with the following property. Fix a closed three-manifold $(M^{3},g)$ and suppose that $\Sigma_{j}\subset (M^{3},g)$ is a sequence of closed embedded minimal surfaces with
\[
\Index(\Sigma_{j})\leq I.
\]
Then, after passing to a subsequence, there is a finite set of points $\cB_{\infty}\subset M$ with $|\cB_{\infty}|\leq I$ and $\varepsilon_{0}>0$ smaller than the injectivity radius of $(M,g)$ so $\cB_{\infty}$ is $4\varepsilon_{0}$-separated, and so that for $\varepsilon \in (0,\varepsilon_{0}]$, if we take $j$ sufficiently large then there exists embedded surfaces $\widetilde \Sigma_{j}\subset (M^{3},g)$ satisfying:
\begin{enumerate}[itemsep=5pt, topsep=5pt]
\item The new surfaces $\widetilde\Sigma_{j}$ agree with $\Sigma_{j}$ outside of $B_{\varepsilon}(\cB_{\infty})$. 
\item The components of $\Sigma_{j}\cap B_{\varepsilon}(\cB_{\infty})$ that do not intersect the spheres $ \partial B_{\varepsilon}(\cB_{\infty})$ transversely and the components that are topological disks appear in $\widetilde\Sigma_{j}$ without any change.
\item The curvature of $\widetilde\Sigma_{j}$ is uniformly bounded, i.e.
\[
\limsup_{j\to\infty}\sup_{x\in\widetilde\Sigma_{j}}|\sff_{\widetilde\Sigma_{j}}|(x) <\infty.
\]
\item Each component of $\widetilde\Sigma_{j}\cap  B_{\varepsilon}(\cB_{\infty})$ which is not a component of $\Sigma_{j}\cap B_{\varepsilon}(\cB_{\infty})$ is a topological disk of area at most $2\pi\varepsilon^{2}(1+o(\varepsilon))$.
\item The genus drops in controlled manner, i.e.,
\[
\genus(\Sigma_{j})-\tilde r(I) \leq \genus(\widetilde\Sigma_{j}) \leq \genus(\Sigma_{j}).
\]
\item The number of connected components increases in a controlled manner, i.e.,
\[
|\pi_{0}(\Sigma_{j})|\leq |\pi_{0}(\widetilde\Sigma_{j})| \leq |\pi_{0}(\Sigma_{j})| + \tilde m(I).
\]
\item While $\widetilde\Sigma_{j}$ is not necessarily minimal, it is asymptotically minimal in the sense that $\lim_{j\to\infty} \Vert H_{\widetilde\Sigma_{j}}\Vert_{L^{\infty}(\widetilde\Sigma_{j})} = 0$.
\end{enumerate}
The new surfaces $\widetilde\Sigma_{j}$ converge locally smoothly to the smooth minimal lamination $\widetilde \cL$ from Theorem \ref{theo:neck}.
\end{coro}

\begin{rema}
The strategy we use to prove Theorem \ref{theo:neck} can be extended to a higher dimensional setting (assuming a uniform volume bound). Certain aspects of the local structure change; in particular, due to the failure of the half-space theorem in higher dimensions, the separation of sheets into ``neck regions'' and ``disk regions'' does not occur in the same way as in three dimensions. As such, we will not attempt to formulate a higher dimensional version of Theorem \ref{theo:neck} or Corollary \ref{coro:snip}, but from the proof of Theorems \ref{theo:fin-top-type-Mn} and \ref{theo:fin-top-type-Rn} it is clear that the general picture described in the introduction holds.
\end{rema}

 \subsection{Related Results} As remarked above, Li--Zhou \cite{LiZhou} have proven compactness results for embedded minimal surfaces with bounded index in a three-manifold with a metric of positive Ricci curvature. This was preceded by the higher dimensional (i.e., allowing the ambient manifold to be $n$-dimensional for $3\leq n \leq 7$) result of Sharp \cite{Sharp}, showing that for a metric of positive Ricci curvature, the space of embedded minimal surfaces with uniformly bounded area and index is compact. 
 
After this paper was completed, we were informed by Carlotto that he had independently arrived at a proof of a slightly different version of Corollary \ref{coro:PSC-bumpy}. His paper \cite{Carlotto:generic} appeared at essentially the same time as ours. 

These works mainly focus on properties of limits of surfaces with uniformly bounded index, rather than the way in which such surfaces degenerate. As such, their arguments are of a rather different nature than those in this paper.

Buzano and Sharp \cite{BuzanoSharp:topHighDim} have given an alternative approach to prove topological bounds for hypersurfaces with bounded index and area, cf.\ Theorem \ref{theo:fin-top-type-Mn}. Their approach also yields further geometric information about the $L^{n}$-norm of the second fundamental form. 

%The version of this article originally posted to the arXiv only discussed the case of ambient three manifolds. After we had completed extending our work to higher dimensions, but before we had updated the arXiv version, we received a preprint \cite{BuzanoSharp:topHighDim} from Buzano and Sharp containing an alternative approach to the topological bounds in higher dimensions, as well as yielding further geometric information about the $L^{n}$-norm of the second fundamental form. 
 
 Finally, we refer to the works of Ros \cite{Ros:compactnessFTC} and Traizet \cite{Traizet:balancing} studying  how complete embedded minimal surfaces in $\RR^{3}$ with bounded total curvature degenerate. Some parallels can be drawn between their results and Theorem \ref{theo:neck}, but our work takes a different technical approach due to the precise behavior of the index.
 
 \subsection{Outline of the paper} 
 
 In Section \ref{sec:prelim}, we make several preliminary definitions and prove curvature bounds away from finitely many points for hypersurfaces with bounded index. Section \ref{sec:ann-decomp} contains a key topological result allowing us to control the topology of the ``intermediate regions'', assuming a curvature bound of the appropriate form. We establish the local picture of degeneration in three-manifolds in Section \ref{sec:3-mfld-degen} and the surgery result in Section \ref{sec:surg-3-mflds}. These results then allow us to establish the three-dimensional compactness results in Section \ref{sec:3-d-compactness}. Section \ref{sec:high-dim} contains the proof of the higher dimensional results. Appendix \ref{app:genus-bdry} contains a discussion of the non-orientable genus as well as the genus of surfaces with boundary. Appendix \ref{app:finite-index-RR3} recalls certain facts about finite index surfaces in $\RR^{3}$, while Appendix \ref{app:two-sided} contains a brief discussion about two-sidedness on small scales. In Appendix \ref{app:remov-sing} we provide proofs of several removable singularity results. Finally, Appendix \ref{app:exam-degen} contains examples to illustrate that the various forms of degeneration discussed in the proof of Propositions \ref{prop:one-point-conc} and \ref{prop:mult-point-conc} can in fact occur. 
  
 \subsection{Acknowledgements} We are grateful to Dave Gabai for several illuminating conversations. We thank Nick Edelen for several useful comments and Alessandro Carlotto for pointing out a mistaken reference in an earlier version of Corollary \ref{coro:PSC-bumpy}. We would like to thank the referee for several useful suggestions. 
 
 O.C.\ would like to acknowledge Jacob Bernstein for a helpful discussion regarding removable singularity results as well as Sander Kupers for kindly answering some questions about topology. He was partially supported by an NSF Graduate Research Fellowship DGE-1147470 during the time which most of this work was undertaken, and also acknowledges support by the EPSRC Programme Grant, ÔSingularities of Geometric Partial Differential EquationsÕ number EP/K00865X/1. D.K. was partially supported by an NSF Postdoctoral Fellowship DMS-1401996. D.M.\ would like to acknowledge Richard Schoen for several helpful conversations. He was partially supported by NSF grant DMS-1512574 during the completion of this work. 
\section{Preliminaries}\label{sec:prelim}

\subsection{Definitions and basic notation} Let $\Sigma$ be a closed embedded minimal hypersurface in $(M,g)$. Recall, whether $\Sigma$ is one-sided or two-sided, the {\it Morse index} of $\Sigma$, henceforth denoted by $\Index(\Sigma)$, is defined as the number of negative eigenvalues of the quadratic form associated to second variation of area:
\[Q (v,v) := \int_\Sigma |\nabla^\perp v|^2 - |\sff_\Sigma|^2|v|^2 - \Ric(v,v)\,d\Sigma, \]
where $v$ is a section of the normal bundle of $\Sigma$ in $M$, $\nabla^\perp$ the induced connection, and $\sff_\Sigma$ the second fundamental form of $\Sigma$. Whenever $\Sigma$ is two-sided, the Morse index is equal to the number of negative eigenvalues of the associated Jacobi operator $\Delta_\Sigma + |\sff_\Sigma|^2 + \Ric(N,N)$ acting on smooth functions $\varphi \in C^\infty (\Sigma)$. If $\Sigma$ is one-sided, however, we consider the orientable double cover $\widehat \Sigma \rightarrow \Sigma $. The corresponding change of sheets involution of $\tau:\widehat \Sigma \rightarrow \widehat\Sigma$ must satisfy $N \circ \tau = -N$ for any choice of unit normal vector $N$ for $\widehat{\Sigma}$. The Morse index of $\Sigma$ is then equal to the negative eigenvalues of the operator $\Delta_{\widehat\Sigma}  + |\sff_{\widehat\Sigma}|^2 + \Ric(N,N) $ over the space of smooth  functions $\varphi \in C^\infty(\widehat\Sigma)$ satisfying $\varphi \circ \tau = -\varphi$. 

In Lemma \ref{lem:two-sided-small-balls}, we show that a hypersurface that is properly embedded in a topological ball is two-sided. We will use this frequently below.

We will be dealing with sequences of embedded minimal surfaces without area bounds in three manifolds and so it will be convenient to consider minimal laminations: A closed set $\cL$ in $M^3$ is called a {\it minimal lamination} if $\cL$ is the union of pairwise disjoint, connected, injectively immersed minimal surfaces, called {\it leaves}. For each point $x\in M^3$, we require the existence of a neighborhood $x\in \Omega$ and a $C^{0,\alpha}$ local coordinate chart $\Phi:\Omega \rightarrow \RR^3$ under which image the leaves of $\cL$ pass through in slices of the form $\RR^2\times\{t\} \cap \Phi(\Omega)$.

All distance functions considered in our work will be induced by some ambient metric, and we denote by $d_h$ the distance function induced by the metric $h$. Given a closed set $\cS$, we let $d_h(\cdot, \cS)$ denote the distance to $\cS$ with respect to the metric $h$. If $\cS$ is a finite set of points, $|\cS|$ will denote its cardinality, and for some $\delta>0$, we say that $\cS$ is {\it$ \delta$-separated} if $d_h(x, \cS\setminus\{x\}) >\delta $ for every $x\in\cS$. We will also consider metric balls and write, as usual, $B_r(p)$ to denote the ball of radius $r>0$ centered at $p$. If a finite set of points is $\delta$-separated and $\delta>r>0$, then the set of points within distance at most $r$ from $\cS$ forms a union of disjoint balls, which we will denote by $B_r(\cS)$ (Note: we omit the dependence of the ambient metric in our notation for $B_r$ as it should always be clear in the context in which is being used). 

\subsubsection{Smooth blow-up sets} The following definition turns out to be quite convenient in the sections to come. Suppose that $(M_{j},g_{j},0_{j})$ is a sequence\footnote{In practice, either $(M_{j},g_{j},0_{j})$ will be a fixed (independently of $j$) compact manifold or $(M_{\infty},g_{\infty},0_{\infty})$ will be Euclidean space.} of complete pointed Riemannian manifolds which are converging in the pointed Cheeger--Gromov sense to $(M_{\infty},g_{\infty},0_{\infty})$. Suppose that $\Sigma_{j}$ is a sequence of embedded minimal hypersurfaces in $(M_{j},g_{j})$. A sequence of finite sets of points $\cB_{j}\subset \Sigma_{j}$ is said to be \emph{a sequence of smooth blow-up sets} if:
	\begin{enumerate}[itemsep=5pt, topsep=5pt]
		\item The set $\cB_{j}$ remains a finite distance from the basepoint $0_{j}$, i.e.
		\[
		\limsup_{j\to\infty}\max_{p\in \cB_{j}} d_{g_{j}}(p,0_{j}) < \infty.
		\]
		\item If we set $\lambda_{j}(p) : = |\sff_{\Sigma_{j}}|(p)$ for $p \in \cB_{j}$, then the curvature of $\Sigma_{j}$ blows up at each point in $\cB_{j}$, i.e.,
		\[
		\liminf_{j\to\infty} \min_{p\in\cB_{j}} \lambda_{j}(p) = \infty.
		\]
		\item If we choose a sequence of points $p_{j}\in \cB_{j}$, then after passing to a subsequence, the rescaled surfaces $\overline\Sigma_{j}:=\lambda_{j}(p_{j})(\Sigma_{j} - p_{j})$ converge locally smoothly to a complete, non-flat, embedded minimal surface $\overline\Sigma_{\infty}\subset\RR^{n}$ without boundary, satisfying
		\[
		|\sff_{\overline\Sigma_{\infty}}|(x) \leq |\sff_{\overline\Sigma_{\infty}}|(0),
		\]
		for all $x\in\RR^n$. 
		\item The blow-up points do not appear in the blow-up limit of the other points, i.e.,
		\[
		\liminf_{j\to\infty}\min_{\substack{p,q \in \cB_{j}\\ p\not=q}} \lambda_{j}(p) d_{g_{j}}(p,q) = \infty.
		\]
	\end{enumerate}

\subsection{Curvature estimates and index concentration} 

Recall that Schoen \cite{Sch83} has proven that two-sided stable minimal surfaces in a three-manifold have uniformly bounded curvature. Subsequently, the two-sided hypothesis was shown to be unnecessary by Ros \cite{Ros:oneSided}.\footnote{We remark that properly \emph{embedded} surfaces are two-sided on small scales (cf.\ Lemma \ref{lem:two-sided-small-balls}), so the curvature estimates from \cite{Sch83} actually suffice when considering embedded surfaces, which is what we will do below.} In particular, we have:
\begin{theo}[\cite{Sch83,Ros:oneSided}]\label{theo:stable-curv-est-3d}
Fix $(M^{3},g)$ a closed three-manifold. There is $C=C(M,g)$ so that if $\Sigma\subset (M,g)$ is a compact stable minimal surface, then 
\[
|\sff_{\Sigma}|(x)\min\{1, d_{g}(x,\partial\Sigma)\} \leq C
\]
for all $x \in \Sigma$.
\end{theo}

Here we show that sequence of embedded minimal surfaces of bounded index have curvature bounds away from at most finitely many points. This can be thought of as a generalization of Schoen and Ros's curvature estimates for stable minimal surfaces to the case of finite Morse index. The proof by induction is most convenient if we prove a more general bound for surfaces with boundary. 

\begin{lemm}\label{lemm:curv-est}
Fix $(M^{3},g)$ a closed three-manifold and $I \in \NN$. Suppose that $\Sigma_{j}\subset (M,g)$ is a sequence of compact embedded minimal surfaces with $\Index(\Sigma_{j})\leq I$. Then, after passing to a subsequence, there exist $C>0$ and a sequence of smooth blow-up sets $\cB_{j}\subset \Sigma_{j}$ with $|\cB_{j}|\leq I$, so that 
\[
|\sff_{\Sigma_{j}}|(x)\min\{1,d_{g}(x,\cB_{j}\cup\partial\Sigma_{j}) \} \leq C.
\]
for all $x \in \Sigma_{j}$.
\end{lemm}

\begin{proof}
We prove this by induction on $I$. When $I = 0$, the surface $\Sigma$ is stable, so the statement is exactly the curvature estimates discussed in Theorem \ref{theo:stable-curv-est-3d}. 

For $I > 0$, consider $\Sigma_{j} \subset (M,g)$ with $\Index(\Sigma_{j}) \leq I$. By passing to a subsequence, we may assume that 
\[
\rho_{j} : = \sup_{x\in \Sigma_{j}} |\sff_{\Sigma_{j}}|(x)\min\{1,d_{g}(x,\partial\Sigma_{j})\} \to \infty.
\]
If we cannot find such a subsequence, it is easy to see that the curvature estimates hold with $\cB_{j} = \emptyset$. 

A standard point picking argument allows us to find $\widetilde p_{j}\in \Sigma_{j}$ so that for $\lambda_{j}=|\sff_{\Sigma_{j}}|(\widetilde p_{j})\to\infty$, the rescaled surfaces 
\[
\overline\Sigma_{j} : = \lambda_{j}(\Sigma_{j}-\widetilde p_{j})
\]
converge locally smoothly, after passing to a subsequence, to a properly embedded\footnote{Usually, the blow-up limit $\widehat\Sigma_{\infty}$ would only be injectively immersed. Here, because it has finite index and no boundary, it is properly embedded by Theorem \ref{theo:fin-index-imp-proper}.} two-sided minimal surface in $\RR^{3}$, $\widehat \Sigma_{\infty}$, of index at most $I$ and with no boundary, so that
\[
|\sff_{\widehat\Sigma_{\infty}}|(x) \leq |\sff_{\widehat\Sigma_{\infty}}|(0) = 1.
\]
For the reader's convenience, we recall the point picking argument at the end of the proof. 

Because $\widehat\Sigma_{\infty}$ is non-flat, there is some radius $ \widehat R > 0$ so that $\widehat\Sigma_{\infty}\cap B_{\widehat R}(0)$ has non-zero index, $\widehat\Sigma_{\infty}\setminus B_{\widehat R}(0)$ is stable, and $\widehat\Sigma_{\infty}$ intersects $\partial B_R(0)$ transversely. Moreover, taking $\widehat R$ larger if necessary, we may arrange that all of these properties are satisfied in addition to
\begin{equation}\label{eq:hat-sig-infty-curv-est}
|\sff_{\widehat\Sigma_{\infty}}|(x) \leq \frac 14.
\end{equation}
for $x \in \widehat\Sigma_{\infty}\setminus B_{\widehat R}(0)$.\footnote{That non-flatness of $\widehat\Sigma_{\infty}$ implies non-zero index is a consequence of \cite{Fischer-Colbrie-Schoen,doCarmoPeng,Pogorelov}. The remaining claims in this paragraph can be proven as follows: because $\hat\Sigma_{\infty}$ has finite index, it has finite total curvature \cite[Theorem 2]{Fischer-Colbrie:1985}. By \cite[Proposition 1]{Schoen:symmetry}, $\widehat\Sigma_{\infty}$ is ``regular at infinity,'' i.e., graphical over a fixed plane outside of a compact set with good asymptotic behavior. This is easily seen to imply the remaining claims. }

We define $\widetilde{\Sigma}_{j} : = \Sigma_{j} \setminus B_{\widehat R/\lambda_{j}}(\widetilde p_{j})$. For $j$ large, this ball cannot intersect the boundary of $\Sigma_{j}$ (by the choice of $\widetilde p_{j}$ and because $\rho_{j}\to\infty$) and $\partial B_{\widehat R/\lambda_{j}}(\tilde p_{j})$ intersects $\Sigma_{j}$ transversely. Thus, $\widetilde{\Sigma}_{j}$ is a smooth compact minimal surface with smooth, compact boundary \[\partial\widetilde{\Sigma}_{j} = \partial\Sigma_{j} \cup (\partial B_{\widehat R/\lambda_{j}}(\widetilde p_{j})\cap \Sigma_{j}).\]
For $j$ large, $ \Index(\widetilde{\Sigma}_{j}) \leq I-1$. By the inductive hypothesis, passing to a subsequence, there is a sequence of smooth blow-up sets $\widetilde\cB_{j} \subset \Sigma_{j}$ with $|\widetilde\cB_{j}|\leq I-1$ and a constant $\widetilde C$ (independent of $j$) so that
\begin{equation}\label{eq:curv-est-induct-hyp}
|\sff_{\widetilde{\Sigma}_{j}}|(x)\min\{ 1,d_{g}(x,\widetilde\cB_{j} \cup \partial\widetilde{\Sigma}_{j})\} \leq \widetilde C.
\end{equation}

We claim that $\cB_{j} : = \widetilde\cB_{j}\cup \{\widetilde p_{j}\}$ is a sequence of smooth blow-up sets. The only thing we must check is that none of the points in $\widetilde\cB_{j}$ can appear in the blow-up around $\widetilde p_{j}$ and vice versa (in particular, this guarantees that rescaling  $\Sigma_{j}$ around points in $\widetilde\cB_{j}$ still yields a smooth limit). First, suppose that 
\[
\liminf_{j\to\infty}\min_{\widetilde r \in \widetilde \cB_{j}} \lambda_{j} d_{g_{j}}(\widetilde r,\widetilde p_{j}) < \infty,
\]
where we recall that $\lambda_{j} = |\sff_{\Sigma_{j}}|(\widetilde p_{j})$. Assume that, the minimum is attained at $\widetilde r_{j}\in\widetilde\cB_{j}$. By choice of $\widehat R$ (specifically \eqref{eq:hat-sig-infty-curv-est}) we see that after passing to a subsequence,
\[
\eta_{j} : = |\sff_{\Sigma_{j}}|(\widetilde r_{j}) \leq \frac 12 |\sff_{\Sigma_{j}}|(\widetilde p_{j}) = \frac 12 \lambda_{j}.
\]
Thus, we have reduced to the other possibility, i.e.
\[
\liminf_{j\to\infty} \eta_{j}d_{g_{j}}(\widetilde r_{j},\widetilde p_{j}) < \infty.
\]
However, this is a contradiction, as the blow-up of $\widetilde\Sigma_{j}$ around $\widetilde r_{j}$ has no boundary (by the inductive step). 

Now, suppose that there is $z_{j} \in \Sigma_{j}$ so that
\[
\limsup_{j\to\infty}|\sff_{\Sigma_{j}}|(z_{j})\min\{1,d_{g_j}(z_{j},\cB_{j}\cup\partial\Sigma_{j})\}  = \infty.
\]
Combined with \eqref{eq:curv-est-induct-hyp} and choice of $\widetilde p_{j}$, we may pass to a subsequence with $z_{j} \in \widetilde\Sigma_{j}$ and
\begin{align*}
d_{g_j}(z_{j},\cB_{j}\cup\partial\Sigma_{j}) & = d_{g_j}(z_{j},\widetilde p_{j}) \to 0,\\
d_{g_j}(z_{j},\widetilde\cB_{j}\cup \partial\widetilde \Sigma_{j}) & = d_{g_j}(z_{j},\widetilde p_{j}) - \frac{\widehat R}{\lambda_{j}}.
\end{align*}
Because $\widehat\Sigma_{\infty}$ has bounded curvature, we see that $z_{j}$ cannot appear in the blow-up around $\widetilde p_{j}$, i.e.,
\[
\liminf_{j\to\infty} \lambda_{j}d_{g_j}(z_{j},\widetilde p_{j}) = \infty.
\]
Thus,
\[
\limsup_{j\to\infty} |\sff_{\Sigma_{j}}|(z_{j}) \frac{\widehat R}{\lambda_{j}} \leq \limsup_{j\to\infty} \frac{\widetilde C \widehat R}{\lambda_{j}d_{g_j}(z_{j},\widetilde\cB_{j}\cup\partial\widetilde\Sigma_{j})} = 0.
\]
Combined with \eqref{eq:curv-est-induct-hyp}, this implies that
\[
\widetilde C \geq \limsup_{j\to\infty} |\sff_{\Sigma_{j}}|(z_{j}) \min\{1,d_{g_j}(z_{j},\widetilde\cB_{j}\cup\partial\widetilde\Sigma_{j}) \} = \infty,
\]
a contradiction. This completes the proof.

Finally, we recall the point-picking argument used above to construct $\hat\Sigma_{\infty}$. Choose $\tilde q_{j} \in \Sigma_{j}$ so that 
\[
|\sff_{\Sigma_{j}}|(\tilde q_{j}) \min\{1,d_{g}(\tilde q_{j},\partial\Sigma_{j})\} = \rho_{j} \to \infty
\]
and set $r_{j} = |\sff_{\Sigma_{j}}|(\tilde q_{j})^{-\frac 12}$. Then, choose $\widetilde p_{j} \in \Sigma_{j}\cap B_{r_{j}}(\widetilde q_{j})$ so that
\[
|\sff_{\Sigma_{j}}|(\widetilde p_{j}) d_{g}(\widetilde p_{j},\partial B_{r_{j}}(\widetilde q_{j})) = \max_{x \in \Sigma_{j}\cap B_{r_{j}}(\widetilde q_{j})}|\sff_{\Sigma_{j}}|(x) d_{g}(x,\partial B_{r_{j}}(\widetilde q_{j})).
\]
Note that the right hand side is at least $|\sff_{\Sigma_{j}}|(\tilde q_{j})^{\frac 12}$ (which is tending to infinity) by choice of $r_{j}$. Let $R_{j} = d_{g}(\widetilde p_{j},\partial B_{r_{j}}(\widetilde q_{j}))$. Because $d_{g}(x,\partial B_{R_{j}}(\widetilde p_{j})) \leq d_{g}(x,\partial B_{r_{j}}(\widetilde q_{j}))$ for $x \in B_{R_{j}}(\widetilde p_{j})$, we find that 
\[
|\sff_{\Sigma_{j}}|(\widetilde p_{j}) d_{g}(\tilde p_{j},\partial B_{R_{j}}(\tilde p_{j})) = \max_{x \in \Sigma_{j}\cap B_{R_{j}}(\widetilde p_{j})}|\sff_{\Sigma_{j}}|(x) d_{g}(x,\partial B_{R_{j}}(\widetilde p_{j})).
\]
Note that $|\sff_{\Sigma_{j}}|(\widetilde p_{j})R_{j} \geq |\sff_{\Sigma_{j}}|(\widetilde q_{j})r_{j} \to \infty$. 

As above, we set $\lambda_{j} = |\sff_{\Sigma_{j}}|(\tilde p_{j})$. Then, the rescaled surfaces 
\[
\overline\Sigma_{j} =\lambda_{j}(\Sigma_{j} - \tilde p_{j})
\]
satisfy
\[
|\sff_{\overline\Sigma_{j}}|(x) d_{\overline g_{j}}(x,\partial B_{\lambda_{j}R_{j}}(0)) \leq \lambda_{j}R_{j},
\]
for $x \in \overline\Sigma_{j}\cap B_{\lambda_{j}R_{j}}(0)$. Thus, if $x\in\overline\Sigma_{j}$ lies in a given compact set of $\RR^{3}$, then
\[
|\sff_{\overline\Sigma_{j}}|(x) \leq \frac{\lambda_{j}R_{j}}{\lambda_{j}R_{j} - d_{\overline g_{j}}(x,0)} \to 1 = |\sff_{\overline\Sigma_{j}}|(0)
\]
as $j\to\infty$. By construction, we see that $d_{\overline g_{j}}(0,\partial \overline \Sigma_{j}) \to \infty$. Passing to a subsequence, we may take a smooth limit of $\lambda_{j}(\Sigma_{j} - \widetilde p_{j})$ to find a complete, non-flat, embedded minimal surface $\widehat \Sigma_{\infty}$ in $\RR^{3}$ of index at most $I$ and with no boundary, completing the point picking argument.
\end{proof}

\begin{coro}\label{coro-curv-bds}
For $(M^{3},g)$ and $I \in \NN$, if $\Sigma_{j}\subset (M,g)$ is a sequence of closed embedded minimal surfaces with $\Index(\Sigma_{j})\leq I$, after passing to a subsequence, there is $C>0$ and a sequence smooth blow-up sets $\cB_{j}\subset \Sigma_{j}$ with $|\cB_{j}|\leq I$, so that
\[
|\sff_{\Sigma_{j}}|(x) d_{g}(x,\cB_{j}) \leq C,
\]
for all $x \in\Sigma_{j}$. 
\end{coro}

In higher dimensions, we similarly have the following curvature estimates.
\begin{lemm}\label{lemm:curv-est-nD}
Fix, for $4\leq n\leq 7$, a closed $n$-dimensional manifold $(M^{n},g)$, as well as $\Lambda > 0$ and $I \in \NN$. Suppose that $\Sigma_{j}\subset (M,g)$ is a sequence of compact embedded minimal hypersurfaces with $\vol_{g}(\Sigma_{j}) \leq \Lambda$ and $\Index(\Sigma_{j})\leq I$. Then, after passing to a subsequence, there exists $C>0$ and a sequence of smooth blow-up sets $\cB_{j}\subset \Sigma_{j}$ with $|\cB_{j}|\leq I$, so that 
\[
|\sff_{\Sigma_{j}}|(x)\min\{1,d_{g}(x,\cB_{j}\cup\partial\Sigma_{j})\} \leq C.
\]
for all $x \in \Sigma_{j}$.
\end{lemm}
\begin{proof}
The argument is similar to Lemma \ref{lemm:curv-est}, so we will be brief. By \cite{Simons:cones,SSY,Schoen-Simon:1981}, if $\hat \Sigma$ is an embedded, two-sided stable minimal hypersuface in $\RR^{n}$ (for $4\leq n\leq 7$), with Euclidean volume growth, i.e. $\lim_{R\to\infty} R^{1-n} \vol(\hat\Sigma\cap B_{R}) < \infty$, then it is a finite union of finitely many parallel planes. Because a complete properly embedded hypersurface in $\RR^{n}$ is two-sided (cf.\ \cite{Samelson}) and an embedded hypersurface in $\RR^{n}$ with bounded second fundamental form and Euclidean volume growth is easily seen to be properly embedded, we obtain the following Bernstein-type result: if $\hat \Sigma$ is an embedded stable minimal hypersurface in $\RR^{n}$ (for $4\leq n\leq 7$) with Euclidean volume growth and bounded second fundamental form, then it is the union of finitely many parallel planes. In particular, we do not need to assume a priori that $\hat\Sigma$ is two-sided.

From this, we obtain the claim when $I=0$. Indeed, if it were false, we could find a sequence of compact embedded stable minimal hypersurfaces $\Sigma_{j}\subset (M^{n},g)$ with $\vol_{g}(\Sigma_{j}) \leq \Lambda$. The point picking argument used above then produces a non-flat, embedded stable minimal hypersurface in $\RR^{n}$ with Euclidean volume growth (by the volume bounds and the monotonicity formula) and uniformly bounded second fundamental form. This contradicts the above observation. 

More generally, assume that the result fails for some fixed index bound $I$ and for a sequence $\Sigma_{j}$. Then, as in the proof of Lemma \ref{lemm:curv-est}, we may find $\tilde p_{j}\in\Sigma_{j}$ so that for $\lambda_{j} = |\sff_{\Sigma_{j}}|(\tilde p_{j})$, the rescaled surfaces 
\[
\overline\Sigma_{j} : = \lambda_{j}(\Sigma_{j}-\widetilde p_{j})
\]
converge locally smoothly, after passing to a subsequence, to an embedded finite index minimal hypersurface $\widehat{\Sigma}_{\infty}$ in $\RR^{n}$ with Euclidean area growth with 
\[
|\sff_{\widehat\Sigma_{\infty}}|(x) \leq |\sff_{\widehat\Sigma_{\infty}}|(0) = 1.
\]
Unlike the case when $n=3$, it is possible that $\hat\Sigma_{\infty}$ has multiple components.\footnote{Recall that the half-space theorem fails for $n\geq 4$. For example, for $n\geq 4$, a catenoid in $\RR^{n}$ is bounded between two parallel planes.} However, the monotonicity formula guarantees that the number of components is bounded. Thanks to the above observation (implying that the index of $\hat\Sigma_{\infty}$ is non-zero) and the fact that there are only finitely many components, we may choose $\hat R$ exactly as in the proof of Lemma \ref{lemm:curv-est-nD}. The rest of the proof proceeds by removing $B_{\hat R/\lambda_{j}}(p_{j})$ from $\Sigma_{j}$ and using the inductive step, exactly as in Lemma \ref{lemm:curv-est}. 
\end{proof}

\section{Annular decomposition from curvature estimates}\label{sec:ann-decomp}

The following lemma is a generalization of \cite[p.\ 251]{Whi87} (see also \cite[Lemma 4.1]{MePeRo13}). It will play a crucial role in later arguments, allowing us to transmit topological information between different scales.

\begin{lemm}[Annular decomposition]\label{lemm:ann-decomp}
There is $0 < \tau_{0}  < \frac 12$ with the following property. Assume that $g$ is a Riemannian metric on $\{|x|\leq 4\}\subset \RR^{n}$ which is sufficiently smoothly close to $g_{\RR^{n}}$. Suppose that $\Sigma\subset B_{2}(0)$ is a properly embedded hypersurface with $\partial\Sigma\subset \partial B_{2}(0)$. Assume that for some $\tau \leq \tau_{0}$ and $p \in B_{\tau_{0}}(0)$, we have:
\begin{enumerate}[itemsep=5pt, topsep=5pt]
\item Each component of $\Sigma$ intersects $B_{\tau}(p)$.
\item The hypersurface $\Sigma$ intersects $\partial B_{\tau}(p)$ transversely in $m$ manifolds diffeomorphic to $\SS^{n-2}$ with the standard smooth structure.
\item The curvature of $\Sigma$ satisfies $|\sff_{\Sigma}|(x)d_{g}(x,p)\leq \frac 14$ for all $x \in \Sigma \cap \left( \overline{B_{1}(0)} \setminus B_{\tau}(p) \right) $. 
\end{enumerate}
Then, $\Sigma$ intersects $\partial B_{1}(0)$ transversely in $m$ manifolds diffeomorphic to the standard $\SS^{n-2}$ and each component of $\Sigma\cap \left(\overline{B_{1}(0)}\setminus {B_{\tau}(p)}\right)$ is diffeomorphic to $\SS^{n-2}\times [0,1]$ with the standard smooth structure.
\end{lemm}

\begin{proof}
As long as $g$ is sufficiently close to $g_{\RR^{n}}$, working in normal coordinates around $p$, a computation as in \cite[pp.\ 417--8]{Huisken-Ilmanen:2001} shows that the third hypothesis implies that the curvature of $\Sigma$ with respect to $g_{\RR^{n}}$ satisfies 
\[
|\sff^{\RR^{n}}_{\Sigma}|(x) d_{\RR^{n}}(x,p) \leq \frac 12.
\]
Hence, it is not hard to check that it is suffices to take $g=g_{\RR^{n}}$. 

Choose $\chi \in C^{\infty}_{c}([0,1))$ a smooth positive cutoff function so that $\chi(r) \in [0,1]$, $\chi(r) = 1$ for $r \leq \frac {1}{4}$ and $\chi(r) = 0$ for $r$ sufficiently close to $1$. We will take $\tau>0$ sufficiently small based on this fixed cutoff function. Consider the function
\[
f(x) = d_{\RR^{n}}(x,p)^{2}\chi(d_{\RR^{n}}(0,x)^{2}) + d_{\RR^{n}}(x,0)^{2}(1-\chi(d_{\RR^{n}}(0,x)^{2})).
\]
By assuming $\tau>0$ is sufficiently small, we see that $f(x) = d_{\RR^{n}}(x,p)^{2}$ near $\partial B_{\tau}(p)$ and $f(x) = d_{\RR^{n}}(x,0)^{2}$ near $\partial B_{1}(0)$. Note that for any point $q \in \RR^{n}$,
\begin{align*}
\nabla^{\Sigma}(d_{\RR^{n}}(x,q)^{2}) & = 2 ((x-q) - \bangle{x-q,N}N)\\
(D^2_{\Sigma} (d_{\RR^{n}}(x,q)^{2}))_x (v,v) & = 2\left(|v|^2-\sff_\Sigma(x)(v,v)\bangle{x-q,N}\right),
\end{align*}
where $N$ is any choice of normal vector at $x$ and $v$ is any vector in $T_x\Sigma$. Thus, we compute
\begin{align*}
(D^{2}_{\Sigma} f)_{x} (v,v) & = 2\chi\left(|v|^2-\sff_\Sigma(x)(v,v)\bangle{x-p,N}\right)\\
& + 2(1-\chi)\left(|v|^2-\sff_\Sigma(x)(v,v)\bangle{x,N}\right)\\
& + 2\chi'(d_{\RR^{n}}(x,p)^{2}-d_{\RR^{n}}(x,0)^{2})\left(|v|^2-\sff_\Sigma(x)(v,v)\bangle{x,N}\right)\\
& + 4\chi''(d_{\RR^{n}}(x,p)^{2}-d_{\RR^{n}}(x,0)^{2})\left( \bangle{x,v} - \bangle{x,N}\bangle{N,v} \right)^{2}\\
& + 8 \chi' \left( \bangle{x-p,v} - \bangle{x-p,N}\bangle{N,v} \right)\left( \bangle{x,v} - \bangle{x,N}\bangle{N,v} \right)\\
& - 8 \chi' \left( \bangle{x,v} - \bangle{x,N}\bangle{N,v} \right)^{2}\\
& = 2\left(|v|^2- \sff_\Sigma(x)(v,v)\bangle{x-p,N}\right)\\
& - 2(1-\chi) \sff_\Sigma(x)(v,v)\bangle{p,N}\\
& + 2\chi'(d_{\RR^{n}}(x,p)^{2}-d_{\RR^{n}}(x,0)^{2})\left(|v|^2-\sff_\Sigma(x)(v,v)\bangle{x,N}\right)\\
& + 4\chi''(d_{\RR^{n}}(x,p)^{2}-d_{\RR^{n}}(x,0)^{2})\left( \bangle{x,v} - \bangle{x,N}\bangle{N,v} \right)^{2}\\
& - 8 \chi' \left( \bangle{p,v} - \bangle{p,N}\bangle{N,v} \right)\left( \bangle{x,v} - \bangle{x,N}\bangle{N,v} \right).
\end{align*}
Observe that for $\tau>0$ sufficiently small, $|\sff_{\Sigma}|(x)| \leq \frac 12$ on the supports of $1-\chi$, $\chi'$ and $\chi''$. In particular, it is easy to see that on $\Sigma\setminus B_{\tau}(p)$,
\[
(D^{2}_{\Sigma}f)_{x}(v,v) \geq 2\left(|v|^2- \sff_\Sigma(x)(v,v)\bangle{x-p,N}\right) - Cd_{\RR^{n}}(p,0) |v|^{2},
\]
for some $C>0$ independent of $\tau$. Combined with the assumed second fundamental form bounds, we have that 
\[
(D^{2}_{\Sigma}f)_{x}(v,v) \geq 2\left(\frac 3 4 -  Cd_{\RR^{n}}(p,0)\right) |v|^{2}.
\]
Thus, as long as $d_{\RR^{n}}(p,0) \leq \tau$ is sufficiently small, this is strictly positive. 

Choosing such a $\tau$, any critical point of $f$ in $\Sigma \setminus B_{\tau}(p)$ must be a strict local minimum. The mountain pass lemma then implies that $f$ cannot have any critical points in the interior of $\Sigma\setminus B_{\tau}(p)$. Thus, the result follows from standard Morse theory.
\end{proof}

\section{Degeneration of bounded index minimal surfaces in three-manifolds}\label{sec:3-mfld-degen}

Let $I$ be a natural number. In this section, we analyze how a sequence of embedded minimal surfaces with index at most $I$ in a three-manifold might degenerate and prove Theorem \ref{theo:neck}. By the curvature estimates from Corollary \ref{coro-curv-bds}, we will be mostly working on small scales near a finite set of at most $I$ points so that we will frequently find ourselves in situations where the following hypothesis, which we will call \makeatletter
 \Hy@raisedlink{\hypertarget{defi:aleph}{}}$(\aleph)$, hold.
 \vspace{7pt} 
 
Suppose that $g_{j}$ is a sequence of metrics on $\{|x| \leq 2r_{j} \} \subset \RR^{3}$ with $r_{j}\to\infty$, so that $g_{j}$ is locally smoothly converging to the Euclidean metric $g_{\RR^{3}}$. Assume also that:
 \makeatother
 
 \begin{enumerate}[itemsep=5pt, topsep=5pt]
 \item We have $\Sigma_{j}\subset B_{r_{j}}(0)$ a sequence of properly embedded minimal surfaces with $\partial\Sigma_{j}\subset \partial B_{r_{j}}(0)$.
 \item The surfaces have $\Index(\Sigma_{j})\leq I$.
 \item There is a sequence of non-empty smooth blow-up sets $\cB_{j}\subset B_{\tau_{0}}(0)$ (where $\tau_{0}$ is fixed in Lemma \ref{lemm:ann-decomp}) with $|\cB_{j}|\leq I$ and $C>0$ so that 
 \[
 |\sff_{\Sigma_{j}}|(x) d_{g_{j}}(x,\cB_{j}\cup\partial\Sigma_{j}) \leq C,
 \]
 for $x \in\Sigma_{j}$.
 \item The smooth blow-up sets converge to a set of points $\cB_{\infty}$ and there is a smooth lamination $\Lambda \subset \RR^{3}\setminus \cB_{\infty}$ so that $\Sigma_{j}$ converges locally smoothly to $\Lambda$ away from $\cB_{\infty}$.
 \end{enumerate}
 \noindent Then, we will say that $\Sigma_{j}$ satisfies \hyperlink{defi:aleph}{$(\aleph)$}. \\ 

Observe that by Lemma \ref{lem:two-sided-small-balls}, the surfaces $\Sigma_{j}$ are all two-sided. We will use this repeatedly below without comment.
 
All the statements we will prove when working under hypothesis \hyperlink{defi:aleph}{$(\aleph)$} will turn out to be open conditions, so the reader may think of all the metric balls to be defined using the Euclidean distance.

\begin{lemm}\label{lemm:lam-lim-planes}
For $\Lambda$ as in \hyperlink{defi:aleph}{$(\aleph)$}, the lamination $\Lambda$ extends across $\cB_{\infty}$ to a smooth lamination $\widetilde\Lambda \subset \RR^{3}$. After a rotation, $\widetilde\Lambda = \RR^{2}\times K$ for $K\subset \RR$ closed. 
\end{lemm}
\begin{proof}
We claim there exists  $\varepsilon > 0$ sufficiently small so that each leaf in $\Lambda \cap B_{\varepsilon}(\cB_{\infty})$ has stable universal cover. We first choose $\varepsilon >0$ sufficiently small so that $\cB_{\infty}$ is $4\varepsilon$-separated (we will choose $\varepsilon>$ smaller below). 

On one hand, if a leaf of $\Lambda \cap B_{\varepsilon}(\cB_{\infty})$ has the convergence to occurring with multiplicity bigger than one, then it must have stable universal cover (cf.\ \cite[Lemma A.1]{MeRo06}). On the other hand, consider the set of leaves of $\Lambda \cap B_{\varepsilon}(\cB_{\infty})$ where the convergence to occurs with multiplicity one. By passing to a double cover if necessary, we may assume that all such leaves are two-sided.\footnote{Because the $\Sigma_{j}$ are all two-sided by Lemma \ref{lem:two-sided-small-balls}, even if they limit to a one-sided leaf, the index bounds hold for the two-sided double cover.} Each leaf must have bounded index, and sum of the index of such leaves must be bounded above by $I$ (or else it would violate the bound for $\Index(\Sigma_j)$). In particular there are only finitely many unstable leaves. For each leaf, we may argue as in \cite[Proposition 1]{Fischer-Colbrie:1985} to find $\varepsilon > 0$ even smaller so that it is stable in  $B_{\varepsilon}(\cB_{\infty})$. Since there are only finitely many of such leaves, we may arrange that each leaf of $\Lambda \cap B_{\varepsilon}(\cB_{\infty})$ has stable universal cover.\footnote{Here we use the well-known fact that two-sided stability passes to covers \cite{Fischer-Colbrie-Schoen}.}

Thus, by Proposition \ref{prop:remov-sing-two-sided-stab-lam}, $\Lambda\cap B_{\varepsilon}(\cB_{\infty})$ extends across $\cB_{\infty}$. Thus, there is a smooth lamination $\widetilde\Lambda\subset \RR^{3}$ with $\Lambda = \widetilde\Lambda \setminus \cB_{\infty}$. Finally, by Corollary \ref{coro:limit-lam-struct}, $\widetilde\Lambda$ is either a non-flat single properly embedded surface of finite total curvature, or else is a lamination of $\RR^{3}$ by parallel planes. 

If $\widetilde\Lambda$ is a non-flat properly embedded surface of finite total curvature, then the convergence of $\Sigma_{j}$ to $\Lambda$ (away from $p_{\infty}$) must occur with multiplicity one.\footnote{To see this, note that if the convergence had multiplicity greater than one, then $\widetilde\Lambda$ would necessarily be stable and thus flat by \cite{Fischer-Colbrie-Schoen,doCarmoPeng,Pogorelov}). This follows by combing e.g., \cite[Lemma A.1]{MeRo06} with the fact that stability extends across isolated points.} Finally, because the convergence occurs with multiplicity one, Allard's regularity theorem\footnote{The version proven in \cite{White:brakke} is also applicable here.} \cite{Allard:varifold} would imply that $\Sigma_{j}$ converged smoothly to $\widetilde\Lambda$ near $\cB_{\infty}$. This is not compatible with the definition of smooth blow-up set, so we see that $\widetilde\Lambda$ is a (non-empty) lamination of $\RR^{3}$ by parallel planes. This completes the proof.
\end{proof}

It is convenient to write  $\Sigma_{j}'$ for the union of components of $\Sigma_{j}\cap B_{2}(0)$ which contain at least one point in $\cB_{j}$, and $\Sigma_{j}''$ for the union of components of $\Sigma_{j}\cap B_{2}(0)$ which contain no points in $\cB_{j}$. Whenever $j$ is sufficiently large, these will represent, respectively, the neck and disk components of $\Sigma_j$.

\begin{lemm}\label{lem:disk-type-have-bd-curv}
Assume that $\Sigma_{j}$ are as in \hyperlink{defi:aleph}{$(\aleph)$}. The surfaces $\Sigma_{j}'' \subset B_{2}(0)$ have uniformly bounded curvature, i.e.
\[
\limsup_{j\to\infty} \sup_{x \in\Sigma_{j}''}|\sff_{\Sigma_{j}}|(x) < \infty.
\]
\end{lemm}
\begin{proof}
After passing to a subsequence, suppose that $z_{j}\in \Sigma_{j}''$ satisfies
\[
|\sff_{\Sigma_{j}}|(z_{j}) = \sup_{x \in \Sigma_{j}''}|\sff_{\Sigma_{j}}|(x) := \lambda_{j}'' \to \infty.
\]
Then, 
\[
\overline{\Sigma_{j}''} : = \lambda_{j}''(\Sigma_{j}''-z_{j})
\]
will converge to a complete\footnote{By (3) in \hyperlink{defi:aleph}{$(\aleph)$} we see that for all $r>0$, for $j$ sufficiently large, $B_{r(\lambda_{j}'')^{-1}}(z_{j})\cap \partial\Sigma_{j}'' =\emptyset$. Note that this also guarantees that we can find such $z_{j}$.} non-flat properly embedded two-sided minimal surface $\overline{\Sigma_{\infty}
''}$ in $\RR^{3}$ with finite index (cf.\ Theorem \ref{theo:fin-index-imp-proper}). On the other hand, we claim that after passing to a subsequence,
\[
\overline{\Sigma_{j}'} : = \lambda_{j}''(\Sigma_{j}' -z_{j}) 
\]
converges away from some finite set of points $\overline{\cB}_{\infty}$ to a non-empty smooth lamination $\overline{\Lambda'}$ of $\RR^{3}\setminus \overline{\cB}_{\infty}$. The reason that $\overline{\Lambda'}$ is non-empty is that 
\[
\limsup_{j\to\infty} \min_{p\in\cB_{j}} \lambda_{j}''d_{g_{j}}(z_{j},p) < \infty
\]
by the curvature estimates assumed in \hyperlink{defi:aleph}{$(\aleph)$}. Thus, at least one point in the rescaled blow-up sets must remain at a bounded distance from the origin. Corollary \ref{coro:lam-min-pts-smooth-or-plane} and then Theorem \ref{theo:fin-index-imp-proper} imply that $\overline{\Lambda'}$ contains either a plane or properly embedded minimal surface with finite total curvature $\overline{\Sigma'_{\infty}}$. 

Because $\overline{\Sigma''_{\infty}}$ is non-flat, the half-space theorem (cf.\ Corollary \ref{coro:limit-lam-struct}) implies that $\overline{\Sigma''_{\infty}} = \overline{\Sigma'_{\infty}}$. However, this implies that $\lambda''_{j}(\Sigma_{j}-z_{j})$ limits to $\overline{\Sigma''_{\infty}}$ with multiplicity greater than one. As in the previous lemma, this contradicts the fact that $\overline{\Sigma''_{\infty}}$ is not flat.
\end{proof}

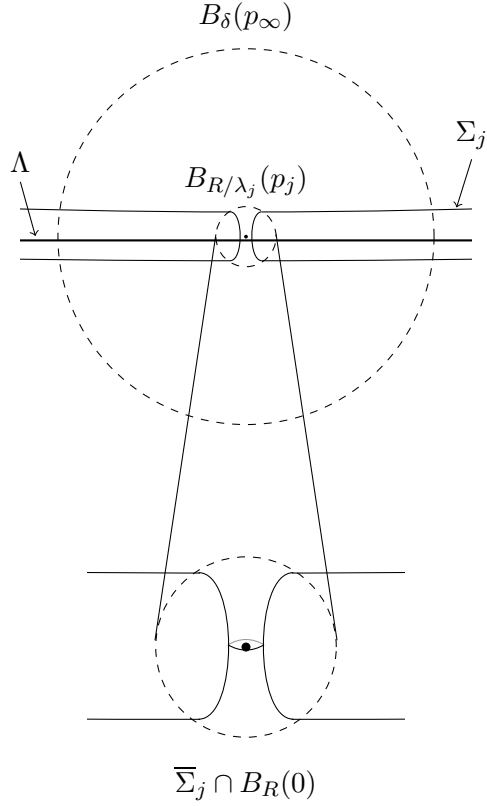
\begin{figure}
\begin{tikzpicture}
\clip (-4,-7.3) rectangle (4,4);

\begin{scope}
	\clip (-3,-4) rectangle (3,4);
	\draw (-5,0) to [bend right = 1] (5,0);
	\draw (-5,.7) to [bend right = 2] (5,.7);	
	\filldraw [white] (-.2,-.3) rectangle (.2,1);
	\draw (-.21,-.05) to [bend right=90, looseness = .7] (-.21,.6);
	\draw (.21,-.05) to [bend left=90, looseness = .7] (.21,.6);
	\filldraw (0,.27) circle (.5pt);
	\draw [dashed] (0,.27) circle (.4);
	\draw [dashed] (0,.27) circle (2.5);
	\draw [thick] (-5,.22) -- (5,.22);
\end{scope}
\begin{scope}[scale = 3, shift = {(0,-2)}]
	\clip (-.7,-1) rectangle (.7,1);
	\draw (-5,0) to [bend right = 1] (5,0);
	\draw (-5,.7) to [bend right = 2] (5,.7);	
	\filldraw [white] (-.2,-.3) rectangle (.2,1);
	\draw (-.21,-.05) to [bend right=90, looseness = .7] (-.21,.6);
	\draw (.21,-.05) to [bend left=90, looseness = .7] (.21,.6);
	\draw (-.08,.28) to [bend right = 30] (.08,.28);
	\draw [opacity = .4] (-.08,.28) to [bend left = 30] (.08,.28);
	\filldraw (0,.27) circle (.5pt);
	\draw [dashed] (0,.27) circle (.4);
\end{scope}

\draw(-.4,.3) -- (-1.21,-5.1);
\draw(.4,.3) -- (1.21,-5.1);

\node at (0,1) {$B_{R/\lambda_{j}}(p_{j}) $};
\node at (0,3.2) {$B_{\delta}(p_{\infty}) $};
\node at (0,-7) {$\overline\Sigma_{j} \cap B_{R}(0)$};

\draw [->] (-3,1) node [above] {$\Lambda$} -- (-2.8,.28);
\draw [->] (3,1.3) node [above] {$\Sigma_{j}$} -- (2.8,.66);

\end{tikzpicture}
\caption{An illustration of the proof of Proposition \ref{prop:one-point-conc}. The crucial estimate \eqref{eq:one-pt-curv-14bds} allows us to transfer topological information from the rescaled picture (on the bottom) to the original scale (on the top).} \label{fig:one-pt}
\end{figure}

We give an example to illustrate the behavior described in the following two propositions in Appendix \ref{app:exam-degen}. The main idea of the proof of Proposition \ref{prop:one-point-conc} is illustrated in Figure \ref{fig:one-pt}.

\begin{prop}[One point of curvature concentration]\label{prop:one-point-conc}
Suppose that $\Sigma_{j}$ satisfies \hyperlink{defi:aleph}{$(\aleph)$} with $|\cB_{j}| = 1$ for each $j$. Then, the lamination $\Lambda$ extends across $\cB_{\infty}$ to a smooth lamination $\widetilde\Lambda\subset \RR^{3}$. Moreover, for $j$ sufficiently large:
\begin{enumerate}[itemsep=5pt, topsep=5pt]
\item The surfaces $\Sigma_j ''$ are disks with uniformly bounded curvature. 
\item The surfaces $\Sigma_{j}'$ intersect $\partial B_{1}(0)$ transversely in at most $\frac{3}{2}(I+1)$ circles. 
\item The surfaces $\Sigma_{j}'$ have genus at most $\frac{3}{2}(I+1)$.
\item The surfaces $\Sigma_{j}'$ have uniformly bounded area, i.e.,
\[
\limsup_{j\to\infty}\area(\Sigma_{j}') < \infty.
\]
\end{enumerate}
After a rotation, $\widetilde\Lambda = \RR^{2}\times K$ for $K\subset \RR$ closed and $\Sigma_{j}'\cap B_{1}(0)$ smoothly converges away from $\cB_{\infty}$ to $B_{1}(0)\cap(\RR^{2}\times \{\eta\})$ with finite multiplicity, for some $|\eta|\leq \frac 12$. 
\end{prop}
\begin{proof}
Let us write $\cB_{j} =\{p_{j}\}$, $\cB_{\infty} = \{p_{\infty}\}$ and $\lambda_{j} : = |\sff_{\Sigma_{j}}|(p_{j})$. Lemma \ref{lemm:lam-lim-planes} shows that the limit lamination $\Lambda$ extends across $\cB_{\infty}$ to a lamination $\widetilde \Lambda$ by parallel planes. 

By the definition of a smooth blow-up set, after passing to a subsequence, the surfaces
\[
\overline \Sigma_{j} : = \lambda_{j} (\Sigma_{j}-p_{j})
\]
converge to $\overline\Sigma_{\infty}\subset \RR^{3}$, a complete, non-flat, properly embedded (and thus two-sided) minimal surface. It has index at most $I$. By\footnote{We remark that it is not strictly necessary to refer to \cite{ChMa14} here. Indeed, one could argue in a similar manner to the proof of Theorem \ref{theo:fin-top-type-Rn} to use our blow-up strategy along with the fact \cite{Fischer-Colbrie:1985} that ``a finite index surface in $\RR^{3}$ cannot have infinite genus'' to prove that there is $C=C(I)$ so that an embedded minimal surface in $\RR^{3}$ with index at most $I$ has at most $C(I)$ genus and ends. Referring to \cite{ChMa14} allows us to avoid such a discussion (and also allows us to obtain functions $m(I)$ and $r(I)$ in Theorem \ref{theo:neck} that are explicitly computable).} \cite{ChMa14}, the genus $g$ and number of ends $r$ of $\overline\Sigma_{j}$ are both bounded by $\frac{3}{2}(I+1)$. Choose $R>0$ so that $\overline\Sigma_{\infty}$ intersects $\partial B_{R}(0)$ transversely and
\begin{equation}\label{eq:one-pt-conc-choice-of-R}
|\sff_{\overline\Sigma_{\infty}}|(x)d_{\RR^{3}}(x,0) < \frac 1 4
\end{equation}
for $x \in \overline\Sigma_{\infty}\setminus B_{R}(0)$.

First, assume that there is $\delta > 0$ so that for $j$ sufficiently large,
\begin{equation}\label{eq:one-pt-curv-14bds}
|\sff_{\Sigma_{j}}|(x) d_{g_{j}}(x,p_{j}) < \frac 1 4
\end{equation}
for $x \in \Sigma_{j} \cap ( B_{\delta}(p_{j})\setminus B_{R/\lambda_{j}}(p_{j}))$. Because $\Sigma_{j}$ is converging away from $p_{\infty}$ to a lamination consisting of planes, this will immediately imply that for $j$ sufficiently large, \eqref{eq:one-pt-curv-14bds} actually holds for all $x \in \Sigma_{j}\cap (B_{2}(0)\setminus B_{R/\lambda_{j}}(p_{j}))$. From this, assertions (2) through (5) follow easily from Lemma \ref{lemm:ann-decomp} (note that we have assumed that $\cB_{j}\subset B_{\tau_{0}}(0)$ and $\tau_{0} < \frac 12$) and the above description of $\overline\Sigma_{\infty}$. 

It remains to prove the crucial fact that we can find $\delta >0$ so that \eqref{eq:one-pt-curv-14bds} holds for $x \in \Sigma_{j}\cap ( B_{\delta}(p_{j})\setminus B_{R/\lambda_{j}}(p_{j}))$. If this failed, we could find $z_{j} \in \Sigma_{j}\cap (B_{2}(0)\setminus B_{R/\lambda_{j}}(p_{j}))$ with $\delta_{j}:=d_{g}(z_{j},p_{j})\to 0$ and
\begin{equation}\label{eq:one-pt-conc-delta-j-contradiction-prop}
|\sff_{\Sigma_{j}}|(z_{j}) \delta_{j} \geq \frac 1 4.
\end{equation}
We now consider 
\[
\hat \Sigma_{j} = \delta_{j}^{-1}(\Sigma_{j}-p_{j}).
\]
and $\hat z_{j} = \delta_{j}^{-1}(z_{j}-p_{j})$. 

Note that the curvature of $\hat\Sigma_{j}$ at the origin cannot be uniformly bounded as $j\to\infty$, as otherwise $\hat\Sigma_{j}$ would limit to a homothety of $\overline\Sigma_{\infty}$. This would contradict the choice of $R$ and $z_{j}$, in particular \eqref{eq:one-pt-conc-choice-of-R}. Hence, $\hat\Sigma_{j}$ satisfies all of the hypotheses of the proposition (with blow-up set $\hat\cB_{j}=\{0\}$). By Lemma \ref{lemm:lam-lim-planes}, $\hat\Sigma_{j}$ converges subsequentially (away from $0$) to a lamination of $\RR^{3}$ by parallel planes. The (scale invariant) curvature estimates in \hyperlink{defi:aleph}{$(\aleph)$} guarantee that 
\[
|\sff_{\hat\Sigma_{j}}|(x) d_{\hat g_{j}}(x,0) \leq C
\]
for, e.g., $x \in B_{2}(0)\cap\hat\Sigma_{j}$. Thus, the convergence of $\hat\Sigma_{j}$ to the lamination by parallel planes takes place smoothly away from $\{0\}$.

 In particular, we find that $|\sff_{\hat\Sigma_{j}}|(\hat z_{j}) \to 0$ (since $\hat z_{j}$ remains a bounded distance away from $0$). This contradicts \eqref{eq:one-pt-conc-delta-j-contradiction-prop} after rescaling. Thus, \eqref{eq:one-pt-curv-14bds} holds for some $\delta > 0$, completing the proof. 
\end{proof}

\begin{prop}[Multiple points of curvature concentration]\label{prop:mult-point-conc}
There are functions $m(I)$ and $r(I)$ so that the following holds. Suppose that $\Sigma_{j}$ satisfies \hyperlink{defi:aleph}{$(\aleph)$}. Then, the lamination $\Lambda$ extends across $\cB_{\infty}$ to a smooth lamination $\widetilde\Lambda\subset \RR^{3}$. Moreover, for $j$ sufficiently large:
\begin{enumerate}[itemsep=5pt, topsep=5pt]
\item The surfaces $\Sigma_j ''$ are minimal disks of uniformly bounded curvature. 
\item The surfaces $\Sigma_{j}'$ intersect $\partial B_{1}(0)$ transversely in at most $m(I)$ circles. 
\item The surfaces $\Sigma_{j}'$ have genus at most $r(I)$.
\item The surfaces $\Sigma_{j}'$ have uniformly bounded area, i.e.,
\[
\limsup_{j\to\infty}\area(\Sigma_{j}') < \infty.
\]
\end{enumerate}
After a rotation, $\widetilde\Lambda = \RR^{2}\times K$ for $K\subset \RR$ closed and $\Sigma_{j}'\cap B_{1}(0)$ converges to $B_{1}(0)\cap (\RR^{2}\times \{\eta_{1},\dots,\eta_{n}\})$ with finite multiplicity for some $|\eta_{i}|\leq \frac 12$.
\end{prop}
\begin{proof}
We will induct on the index bound $I$ in \hyperlink{defi:aleph}{$(\aleph)$}. If $I = 1$, then the proposition follows from Proposition \ref{prop:one-point-conc} above. Now, assume that the proposition holds for $I-1$ and that $\Index(\Sigma_{j})\leq I$. Lemma \ref{lemm:lam-lim-planes} implies that the lamination $\Lambda$ extends across $\cB_{\infty}$ to $\widetilde\Lambda$ a lamination by parallel planes. 

We first consider the case that $|\cB_{\infty}|\geq 2$. Pick $\delta>0$ so that $\cB_{\infty}$ is $4\delta$-separated. In particular, $B_{\delta}(\cB_{\infty})$ is a disjoint union of balls and, after passing to a subsequence, we may assume that for any connected component $B$ of $B_{\delta}(\cB_{\infty})$,
\[
\Index(\Sigma_{j}\cap B) \geq 1.
\] 
Because we are assuming that $|\cB_{\infty}|\geq 2$, this implies that
\[
\Index(\Sigma_{j}\cap B)\leq I-1.
\]

Now, we choose $\varepsilon_{j}\to 0$ sufficiently slowly so that $\cB_{j}\subset B_{\varepsilon_{j}\tau_{0}/2}(\cB_{\infty})$ and
\[
\liminf_{j\to\infty}\varepsilon_{j} \min_{p\in\cB_{j}} |\sff_{\Sigma_{j}}|(p) = \infty.
\] 
We claim that (after taking $\delta>0$ smaller if necessary) for $j$ large we have
\begin{equation}\label{eq:multi-pt-curv-14bds}
|\sff_{\Sigma_{j}}|(x) d_{g_{j}}(x,\cB_{\infty}) < \frac 1 4
\end{equation}
for $x \in \Sigma_{j}\cap (B_{\delta}(\cB_{\infty})\setminus B_{\varepsilon_{j}}(\cB_{\infty}))$. If this were to fail, then we may argue as in the one point case,\footnote{Observe that things are slightly different than the one point case. In this situation, we work away from the ball $B_{\varepsilon_{j}}(p_{\infty})$ with $p_{\infty}$ fixed, rather than the ball $p_{j}$ as in the one-point case. We must do this to handle the possibility that multiple points in $\cB_{j}$ are converging to the single point $p_{\infty}$.} Proposition \ref{prop:one-point-conc}: after passing to a subsequence, we may choose $p_{\infty}\in\cB_{\infty}$ and $z_{j}\in \Sigma_{j}\cap B_{\delta}(p_{\infty})$ with $z_{j}\to p_{\infty}$, so that 
\[
|\sff_{\Sigma_{j}}|(z_{j})d_{g}(z_{j},p_{\infty}) \geq \frac 1 4,
\] 
for $\delta_{j} = d_{g}(z_{j},p_{\infty})\to 0$. The surfaces
\[
\delta_{j}^{-1}(\Sigma_{j}\cap B_{\delta}(p_{\infty}) - p_{\infty}) 
\]
satisfy the inductive hypothesis and have unbounded curvature. Hence, after passing to a subsequence, they converge to a lamination of $\RR^{3}$ by parallel planes (smoothly near $\partial B_{1}(0)$), contradicting the choice of $z_{j}$. 

Because we now know that \eqref{eq:multi-pt-curv-14bds} holds, we are able to transfer topological information from the scale $B_{\varepsilon_{j}}(\cB_{\infty})$ (where we may apply the inductive hypothesis, by choice of $\varepsilon_{j}$) to the scale $B_{\delta}(\cB_{\infty})$, using Lemma \ref{lemm:ann-decomp}. In particular, we see that any component of $\Sigma_{j}\cap B_{\delta}(\cB_{\infty})$ containing some point in $\cB_{j}$ intersects $B_{\delta}(\cB_{\infty})$ transversely in at most $m(I-1)$ circles, has genus at most $r(I-1)$. On the other hand, observe that $\Sigma_{j}\cap (B_{1}(0)\setminus B_{\delta/2}(\cB_{\infty}))$ is converging smoothly to the planar domains $\widetilde \Lambda \cap (B_{1}(0) \setminus B_{\delta/2}(\cB_{\infty}))$.

From this, it is not hard to check that $\Sigma_{j}'\cap(B_{1}(0)\setminus B_{\delta}(\cB_{\infty}))$ satisfies the hypothesis of Lemma \ref{lemm:adding-genus-ends}, which yields the asserted bounds on the genus and number of boundary components of $\Sigma_{j}'$. Combined with fact that $\widetilde\Lambda$ is a lamination by parallel planes, this also yields the asserted area bounds, so have proven assertions (2)-(5). Finally, assertion (1) follows from Lemma \ref{lem:disk-type-have-bd-curv} and the fact that $\tilde\Lambda$ consists of parallel planes. This completes the proof in the $|\cB_{\infty}|\geq 2$ case.

Thus, it remains to consider the case that $|\cB_{\infty}|=1$. Passing to a subsequence, we may assume that $|\cB_{j}|\geq 2$ for each $j$, as otherwise we could apply Proposition \ref{prop:one-point-conc}. Then, we may choose $p_{j},q_{j}\in \cB_{j}$ so that
\[
\varepsilon_{j}\tau_{0}/2 : = d_{g_{j}}(p_{j},q_{j}) = \max_{\substack{p,q \in \cB_{j}\\ p\not=q}} d_{g_{j}}(p,q) \to 0.
\]
Then, consider the sequence
\[
\overline\Sigma_{j} := \varepsilon_{j}^{-1}(\Sigma_{j}-p_{j})
\]
By definition of a sequence of smooth blow-up sets (i.e., the various points cannot appear in the blow-up of the other points), the curvature must still be blowing up at each point in $\overline\cB_j=\varepsilon_{j}^{-1}(\cB_{j}-p_{j})$. Thus $\overline\Sigma_{j}$ satisfies the hypothesis of the proposition with $|\overline\cB_{\infty}|\geq 2$, so the conclusion of the proposition holds for $\overline\Sigma_{j}$. At this point, we may argue as above (cf.\ the analogous argument in the proof of Proposition \ref{prop:one-point-conc}), establishing the curvature estimate \eqref{eq:multi-pt-curv-14bds} for $x\in \Sigma_{j}\cap (B_{\delta}(p_{j})\setminus B_{\varepsilon_{j}}(p_{j}))$ for some $\delta > 0$. As before, this allows us to remove the singularities in the limit lamination and conclude that it must be a lamination by planes. Using this, we may readily transfer the topological information out to the scale of $B_{1}(0)$ for $\Sigma_{j}$ using Lemma \ref{lemm:ann-decomp} and as before conclude assertions (1) through (5).
\end{proof}

\subsection{Completing the proof of Theorem \ref{theo:neck}} Assume that $\Sigma_{j}\subset (M^{3},g)$ is a sequence of closed embedded minimal surfaces with 
\[
\Index(\Sigma_{j})\leq I.
\]
Passing to a subsequence, Corollary \ref{coro-curv-bds} yields a sequence of smooth blow-up sets $\cB_{j}$ so that $|\cB_{j}|\leq I$ and a constant $C>0$ so that 
\[
|\sff_{\Sigma_{j}}|(x)\min\{1,d_{g}(x,\cB_{j})\} \leq C.
\]
Passing to a further subsequence, $\cB_{j}$ converges to a finite set of points $\cB_{\infty}$ and $\Sigma_{j}$ converges away from $\cB_{\infty}$ to a lamination $\cL \subset M \setminus \cB_{\infty}$. The remaining argument is very similar to the proof of Proposition \ref{prop:mult-point-conc}, so we omit some of the details below. Arguing as in Lemma \ref{lemm:lam-lim-planes}, the lamination $\cL$ extends across $\cB_{\infty}$ to a smooth lamination $\widetilde\cL\subset M$. 

Choose $\varepsilon_{j}\to 0$ sufficiently slowly so that $\cB_{j}\subset B_{\varepsilon_{j}\tau_{0}/2}(\cB_{\infty})$ and
\[
\liminf_{j\to\infty}\varepsilon_{j} \min_{p\in\cB_{j}} |\sff_{\Sigma_{j}}|(p) = \infty.
\]
We claim that by taking $\varepsilon_{0}>0$ sufficiently small (in particular, so that it is smaller than the injectivity radius and so that $\cB_{\infty}$ is $4\varepsilon_{0}$-separated), for $j$ large, we have the improved curvature bounds
\begin{equation}\label{eq:14-curv-bds-final-pf-neck}
|\sff_{\Sigma_{j}}|(x) d_{g}(x,\cB_{\infty}) < \frac 14
\end{equation}
for $x\in \Sigma_{j}\cap (B_{2\varepsilon_{0}}(\cB_{\infty})\setminus B_{\varepsilon_{j}}(\cB_{\infty}))$. To prove this we argue exactly as before: we may pick a connected component $\Sigma_{j}\cap (B_{2\varepsilon_{0}}(\cB_{\infty})\setminus B_{\varepsilon_{j}}(\cB_{\infty}))$ where it fails and rescale a sequence of points where where \eqref{eq:14-curv-bds-final-pf-neck} fails to unit scale. This rescaled sequence then satisfies the hypothesis of Proposition \ref{prop:mult-point-conc}, so it limits to a lamination of $\RR^{3}$ by parallel planes (away from a discrete set). This contradicts the fact that we chose points volating \eqref{eq:14-curv-bds-final-pf-neck}. Thus, we may find some $\varepsilon_{0}>0$ as claimed. Taking $\varepsilon_{0}>0$ even smaller if necessary, we may arrange that for every component $B$ of $B_{2\varepsilon_{0}}(\cB_{\infty})$, the metric $g$ restricted to $B$ and rescaled by by $\varepsilon_{0}^{-1}$ around its center satisfies the hypothesis in Lemma \ref{lemm:ann-decomp}.

Now, Propositions \ref{prop:one-point-conc} and \ref{prop:mult-point-conc} applied to each component $B$ of $\Sigma_{j}\cap B_{\varepsilon_{j}}(\cB_{\infty})$, after rescaling it by $\varepsilon_{j}^{-1}$ around the center of $B$ yields the desired topological information at the scale of $B_{\varepsilon_{j}}(\cB_{\infty})$. The improved curvature estimates in \eqref{eq:14-curv-bds-final-pf-neck} and Lemma \ref{lemm:ann-decomp} then allow us to transfer this information out to the scale of $B_{\varepsilon_{0}}(\cB_{\infty})$, exactly as in the proof of Propositions \ref{prop:one-point-conc} and \ref{prop:mult-point-conc}. In particular, topological statements (1.a), (1.b), (1.c) follow, and also (2.a) by Lemma \ref{lem:disk-type-have-bd-curv}. 

Finally, fix $\varepsilon \in (0,\varepsilon_0]$ and $B = B_{\varepsilon_{0}}(p_{\infty})$ a connected component of $B_{\varepsilon_{0}}(\cB_{\infty})$. Because $\widetilde\cL$ is smooth in $B$, as $\varepsilon \to 0$, each leaf in $\varepsilon^{-1}(\widetilde\cL \cap B - p_{\infty})$ converges with multiplicity one to a plane in $\RR^{3}$. Rotating a local coordinate frame, we may assume that all such planes are of the form $\RR^{2}\times \{t\}$ for some $t\in \RR$.

 Thus, by (1.a) we can see that for $j$ sufficiently large (depending on $\varepsilon$) any component of $\Sigma_{j}\cap B_{2\varepsilon}(p_{\infty})$ must intersect $B_{\varepsilon}(p_{\infty})\setminus B_{\varepsilon/2}(p_{\infty})$ union of at most $m(I)$ annuli, which converge graphically to the annulus $\left( \RR^{2}\times \{0\} \right) \cap \left( B_{\varepsilon}(p_{\infty})\setminus B_{\varepsilon/2}(p_{\infty}) \right)$. Combined with the monotonicity formula, the area estimate (1.d) easily follows. The argument for (2.b) follows a similar line of reasoning, except any disk region is converging smoothly everywhere to a leaf in $\widetilde\cL \cap B$, which is nearly planar on small scales.

\section{Surgery for bounded index surfaces in three-manifolds}\label{sec:surg-3-mflds}

In this section, we describe how Corollary \ref{coro:snip} follows from Theorem \ref{theo:neck}. We first prove the following local description of the surgery operation. 
\begin{prop}[Local picture of surgery]\label{prop:local-surg}
Suppose that $\Gamma_{j}$ is a sequence of embedded surfaces in $B_{3}(0)$ with $\partial\Gamma_{j}\subset \partial B_{3}(0)$, and so that:
\begin{enumerate}[itemsep=5pt, topsep=5pt]
\item The surfaces $\Gamma_{j}\setminus \overline {B_{1}(0)}$ converge smoothly, with finite multiplicity, to the flat annulus $A(3,1) : = \left(B_{3}(0)\setminus\overline{B_{1}(0)}\right)\cap\{x^{3}=0\}$ as $j\to\infty$. 
\item The set of components of $\Gamma_{j}$ which are topological disks converge smoothly to the flat disk $D(3) : = B_{3}(0)\cap\{x^{3}=0\}$ as $j\to\infty$. 
\end{enumerate}
Then, for $j$ sufficiently large, we may construct embedded surfaces $\widetilde\Gamma_{j}$ with $\partial\widetilde\Gamma_{j}\subset \partial B_{3}(0)$, and so that:
\begin{enumerate}[itemsep=5pt, topsep=5pt]
\item The surfaces $\widetilde \Gamma_{j}$ agree with the $\Gamma_{j}$ near $\partial B_{3}(0)$.
\item Any component of $\Gamma_{j}$ which is topologically a disk is unchanged.
\item The surfaces $\widetilde\Gamma_{j}$ converge smoothly, with finite multiplicity, to the flat disk $D(3)$ as $j\to\infty$. 
\end{enumerate}
\end{prop}
\begin{proof}
Fix a smooth cutoff function $\chi:\RR^{2}\to [0,1]$, with $\chi(x) \equiv 1$ for $|x|\geq \frac 74$ and $\chi(x) \equiv 0$ for $|x|\leq \frac 54$.

We define the cylinder and annular cylinder 
\begin{align*}
C(r) &: = \{(x^{1},x^{2},x^{3}:(x^{1})^{2}+(x^{2})^{2} < r^{2}\}, \quad r>0,\\
C(r_{1},r_{2}) &: = C(r_{1})\setminus \overline{C(r_{2})},\quad r_1>r_2>0.
\end{align*}
Taking $j$ sufficiently large, each component of $\Gamma_{j}\cap C(2,1)$ is graphical over the flat annulus $A(2,1)$, and the topological disk components of $\Gamma_{j}\cap C(2)$ are graphical over the flat disk $D(2)$.

For now, we assume at most two of the components of $\Gamma_{j}$ are topological disks, and each of the disk components, if they exist, is either the topmost component or bottommost component. Choose a smooth function $w_{j}: D(2) \to \RR$ so that 
\begin{enumerate}[itemsep=5pt, topsep=5pt]
\item The graph of $w_{j}$ is contained in $B_{3}(0)$.
\item The graph of $w_{j}$ lies strictly above (resp.\ below) the bottommost (resp.\ uppermost) disk if it exists.
\item The function $w_{j}$ converges smoothly to $0$ as $j\to\infty$. 
\end{enumerate}
For example, when $\Gamma_{j}$ contains both an uppermost and bottommost disk, then we may take the average of their respective graphs. We additionally choose real numbers $\eta_{j}\to 0$ so that the graph of $w_{j}+\eta$ satisfies the above properties as well for all $\eta \in (0,\eta_{j})$.

We may find functions $u_{j,1},\dots,u_{j,n(j)}:A(2,1) \to \RR$ so that any non-disk component of $\Gamma_{j}$ is the graph of the $u_{j,l}$ in $C(2,1)$. By assumption, for all $k$,
\[
\sup_{l \in \{1,\dots,n(j)\}} \Vert u_{j,l}\Vert_{C^{k}(A(2,1))} \to 0
\] 
as $j\to\infty$. By embeddedness of $\Gamma_{j}$, we may arrange that
\[
u_{j,1}(x) < u_{j,2}(x) < \dots < u_{j,n(j)}(x)
\]
for $x \in A(2,1)$. 

Now, we define
\[
\widetilde u_{j,l}(x) = \chi(x) u_{j,l}(x) + (1-\chi(x))\left( w_{j}(x) + \frac{l}{n(j)} \eta_{j} \right)
\]
We now define a surface $\widetilde\Gamma_{j}$ which agrees with $\Gamma_{j}$ in $B_{3}(0)\setminus C(2)$ and which is defined inside of $C(2)$ to be the union of the graphs of the $\widetilde u_{j,l}$ along with the disk components in $\Gamma_{j}$, if they exist. It is easy to check that $\widetilde\Gamma_{j}$ satisfies the asserted properties. 

Finally, we may easily reduce the case of general $\Gamma_{j}$ to the above case by considering contiguous subsets of the components of $\Gamma_{j}$ which are in the above form and applying the argument above to the maximal such subsets. This choice will preserve embeddedness, because we have chosen them so that there will at least be a disk separating the non-disk components of different subsets. 
\end{proof}

Now, we may complete the proof of the surgery result. Consider $\Sigma_{j}$ a sequence of compact embedded minimal surfaces in $(M^{3},g)$ with $\Index(\Sigma_{j})\leq I$. We pass to a subsequence so that the conclusion of Theorem \ref{theo:neck} applies. In particular, there is a finite  set of points $\cB_{\infty}\subset M$ with $|\cB_{\infty}|\leq I$ and a smooth lamination $\widetilde\cL$ of $M$ so that $\Sigma_{j}$ converges to $\cL=\widetilde\cL\setminus\cB_{\infty}$ away from $\cB_{\infty}$.

Take $\varepsilon_{0}$ as in Theorem \ref{theo:neck} and choose $\varepsilon \in (0,\varepsilon_{0}]$. Pick any $p\in\cB_{\infty}$; we will show how to perform the surgery in $B_{\varepsilon}(p)$. Write $L$ for the leaf of $\widetilde \cL \cap B_{\varepsilon}(p)$ that passes through $p$. We may fix a diffeomorphism of $\Psi: B_{\varepsilon}(p) \to B_{3}(0)\subset \RR^{3}$ so that $\Psi$ maps $B_{\varepsilon/3}(0)$ difeomorphically onto $B_{1}(0)$ and $L$ onto the flat disk $D(3)\subset \RR^{3}$ as in Proposition \ref{prop:local-surg}. 

Consider the connected components of $\Sigma_{j}\cap B_{\varepsilon}(p_{i})$ which are converging smoothly to $L$ in the annulus $B_{\varepsilon}(p)\setminus \overline{B_{\varepsilon/3}(p)}$ (by Theorem \ref{theo:neck}, this includes all of the neck components, i.e., all of the components of $\Sigma_{j}\cap B_{\varepsilon}(p)$ containing some point in $\cB_{j}$). Using the maximum principle, the area bounds and curvature estimates for the disk components show that they converge smoothly to $L$ (although they might do so with infinite multiplicity). Now, we define $\Gamma_{j}$ to be the union of all of the neck components of $\Sigma_{j}\cap B_{\varepsilon}(p)$, as well as all of the disc components which are directly adjacent (either above or below) to a neck component. 

It is not hard to see that if the uppermost (resp.\ lowermost) component of $\Gamma_{j}$ is a neck component, we may simply add in a disk which is above (resp.\ below) all of the components of $\Gamma_{j}$, but which is below (resp.\ above) all of the disk components not converging to $L$. 

Now we apply Proposition \ref{prop:local-surg} to $\Gamma_{j}$ (and then removing the extra disks on top and bottom, if we had to add them) and replace $\Sigma_{j}\cap B_{\varepsilon}(p)$ by the resulting surface. Repeating this for each $p\in\cB_{\infty}$ yields $\widetilde\Sigma_{j}$. The asserted properties of $\widetilde\Sigma_{j}$ follow easily from Proposition \ref{prop:local-surg} and Theorem \ref{theo:neck}.

\section{Proofs of the three-dimensional compactness results}\label{sec:3-d-compactness}

\begin{proof}[Proof of Theorem \ref{theo:fin-top-type-Mn} for $n=3$]

Fix $I \in \NN$, $A<\infty$, and a closed Riemannian three-manifold $(M,g)$. Suppose that $\Sigma_{j}\subset (M,g)$ is a sequence of connected, embedded, closed minimal surfaces with $\Index(\Sigma_{j})\leq I$ and $\area(\Sigma_{j})\leq A$ but $\genus(\Sigma_{j})\to \infty$. By\footnote{Note one could also prove Theorem \ref{theo:fin-top-type-Mn} for $n=3$ using Theorem \ref{theo:neck} directly (somewhat like we will do for Theorems \ref{theo:fin-top-type-Mn} or \ref{theo:fin-top-type-Rn}).} Corollary \ref{coro:snip}, we may find $\widetilde\Sigma_{j}$ with uniformly bounded area and curvature, but so that 
\[
\genus(\widetilde\Sigma_{j}) \geq \genus(\Sigma_{j}) - \tilde r(I) \to \infty.
\]
This is a contradiction: after passing to a subsequence, the surfaces $\widetilde\Sigma_{j}$ must converge smoothly and with finite multiplicity to some closed, embedded minimal surface $\widetilde\Sigma_{\infty}$ (which must have finite genus). 
\end{proof}

\begin{proof}[Proof Theorem \ref{theo:area-genus-bd-PSC}]
Fix $I \in \NN$ and $(M,g)$ a closed Riemannian three-manifold with positive scalar curvature. We only need to prove the area bound, since the genus bound would immediately follow from Theorem \ref{theo:fin-top-type-Mn} (the case $n=3$ is proven above). Suppose that $\Sigma_{j}\subset (M,g)$ is a sequence of connected, closed, embedded minimal surfaces with $\Index(\Sigma_{j}) \leq I$ and $\area_{g}(\Sigma_{j})\to\infty$. 

After passing to a subsequence, by Theorem \ref{theo:neck}, there is a finite set of points $\cB_{\infty}$ and a lamination $\widetilde\cL\subset M$ so that $\Sigma_{j}$ converges locally to the lamination $\cL:=\widetilde\cL\setminus \cB_{\infty}$ away from $\cB_{\infty}$. Because the area of $\Sigma_{j}$ is diverging, passing to a further subsequence, there is $p \in M\setminus \cB_{\infty}$ so that 
\[
\liminf_{j\to\infty} \area_{g}(\Sigma_{j}\cap B_{r}(p)) = \infty. 
\]
for all $r>0$. A standard argument along the lines of \cite[Lemma 1.1]{MeRo05}, \cite[Lemma A.1]{MeRo06}, and \cite[Proposition 2.1]{CCE} shows that there is a leaf $p \in L \subset \cL$ with stable universal cover and so that for $r>0$ fixed sufficiently small, $\Sigma_{j}\cap B_{r}(p)$ consists of $n(j)\to\infty$ sheets, which are all smoothly graphically converging to $L\cap B_{r}(p)$. 

Because $\cL=\widetilde\cL\setminus\cB_{\infty}$ has removable singularities, there is a smooth complete minimal surface $\widetilde L$ so that $L = \widetilde L \setminus \cB_{\infty}$. The log-cutoff trick shows that stability extends across isolated points, so $\widetilde L$ has stable universal cover $\widehat L$. We  think of $\widehat L$ as an immersed stable minimal surface in $M$. If we consider a disk $D \subset \widehat L$ and if $x$ is any point in the interior of $D$, by Schoen-Yau \cite{ScYa82,ScYa83}, the intrinsic distance to the boundary must satisfy: 
\[d_{D}(x,\partial D)\leq \frac{2\pi\sqrt{2}}{\sqrt{3\kappa_0}},\]
where $\kappa_0>0$ is the infimum of the scalar curvature of $M$. This implies that $\widehat L$ must be compact, since $D$ is arbitrary. By \cite[Theorem 3]{Fischer-Colbrie-Schoen}, $\widehat L$ is a two-sphere. 

We choose $\varepsilon>0$ smaller than $\varepsilon_{0}$ from the surgery theorem and small enough so that $p \not \in \cB_{2\varepsilon}(\cB_{\infty})$. Let $\widetilde\Sigma_{j}$ denote the surfaces resulting from a surgery at scale $\varepsilon$, as constructed in Corollary \ref{coro:snip}. Because the original surfaces $\Sigma_{j}$ are connected, Corollary \ref{coro:snip} implies that the number of components of $\widetilde\Sigma_{j}$ is uniformly bounded above, $|\pi_{0}(\widetilde\Sigma_{j})| \leq m(I) + 1$.

Putting these facts together, we may find a connected component $\widehat \Sigma_{j} \subset \widetilde\Sigma_{j}$ so that $\area_{g}(\widehat\Sigma_{j})\to\infty$ and so that $\widehat\Sigma_{j}\cap B_{\varepsilon}(p)$ is smoothly converging to $L \cap B_{\varepsilon}(p)$. The maximum principle then implies that $\widehat\Sigma_{j}$ converges locally smoothly to $\widetilde L$. In particular, the universal cover of $\widehat\Sigma_{j}$ converges in the sense of immersions to $\widehat L$, which we have seen is a topological sphere. This implies that the area of $\widehat\Sigma_{j}$ is uniformly bounded, a contradiction.  
\end{proof}

\section{Bounded diffeomorphism type in higher dimensions}\label{sec:high-dim}

Here, we discuss the $4\leq n \leq 7$ case of Theorems \ref{theo:fin-top-type-Mn} and \ref{theo:fin-top-type-Rn}. Motivated by the three-dimensional case, we define the hypothesis \makeatletter
 \Hy@raisedlink{\hypertarget{defi:beth}{}}$(\beth)$
 \vspace{7pt} 
as follows. 

\noindent
Fix $4\leq n\leq 7$ and suppose that $g_{j}$ is a sequence of metrics on $\{|x| \leq 2r_{j}\} \subset \RR^{n}$ that is locally smoothly converging to $g_{\RR^{n}}$. Assume that:

 \begin{enumerate}[itemsep=5pt, topsep=5pt]
 \item We have $\Sigma_{j}\subset B_{r_{j}}(0) \subset \RR^{n}$  a sequence of properly embedded minimal hypersurfaces with $\partial\Sigma_{j}\subset \partial B_{r_{j}}(0)$.
 \item The surfaces $\Sigma_{j}$ are connected.
 \item The hypersurfaces have $\Index(\Sigma_{j})\leq I$.
 \item The hypersurfaces satisfy $\vol(\Sigma_{j}) \leq \Lambda r_{j}^{n-1}$.
 \item There is a sequence of non-empty smooth blow-up sets $\cB_{j}\subset B_{\tau_{0}}(0)$ (where $\tau_{0}$ is fixed in Lemma \ref{lemm:ann-decomp}) with $|\cB_{j}|\leq I$ and $C>0$ so that 
 \[
 |\sff_{\Sigma_{j}}|(x) d_{g_{j}}(x,\cB_{j}\cup\partial\Sigma_{j}) \leq C,
 \]
 for $x \in\Sigma_{j}$.
 \item The smooth blow-up sets converge to a set of points $\cB_{\infty}$ and for any $r>0$, the hypersurfaces $\Sigma_{j}\cap B_{r}(0)$ converge in sense of varifolds to a disk with multiplicity $k\in \NN$, i.e. 
 \[
[\Sigma_{j}\cap B_{r}(0)] \rightharpoonup k [\{x^{n}=0\}\cap B_{r}(0)].
\]
 \end{enumerate}

\noindent Then, we say that $\Sigma_{j}$ satisfies \hyperlink{defi:beth}{$(\beth)$}. 

\ \\
Let us briefly note that the main difference between hypothesis \hyperlink{defi:beth}{$(\beth)$} and the hypothesis \hyperlink{defi:alep}{$(\aleph)$} used in three dimensions is the assumption that the surfaces are connected (in addition to the assumption that they satisfy a uniform area bound). The connectedness assumption is useful to compensate for the fact that the half-space theorem fails in higher dimensions. To exploit this assumption, we will work ``big to small'' when proving the crucial curvature estimates, e.g., \eqref{eq:curv-14-nD-1pt}.

\begin{prop}
Given a sequence $\Sigma_{j}$ satisfying \hyperlink{defi:beth}{$(\beth)$} that intersect $\partial B_{1}(0)$ transversely, we may pass to a subsequence so that all of the $\Sigma_{j}\cap B_{1}(0)$ are diffeomorphic. 
\end{prop}
\begin{proof}
We prove this by induction on $I$. For $I=0$ this trivially follows from the curvature and area estimates. 

We first consider the one point of concentration, i.e. $|\cB_{j}| = 1$. We write $\cB_{j} = \{p_{j}\}$ and $\cB_{\infty}=\{p_{\infty}\}$ and $\lambda_{j} = |\sff_{\Sigma_{j}}|(p_{j})$. By passing to a subsequence, we have that 
\[
\overline \Sigma_{j} : = \lambda_{j}(\Sigma_{j}-p_{j})
\]
converges to $\overline\Sigma_{\infty}\subset \RR^{n}$ a complete, non-flat, properly embedded minimal surface with index at most $I$ and $\vol(\overline{\Sigma}_{\infty} \cap B_{r}(0))\leq \Lambda r^{n-1}$ (by the monotonicity formula). Because of these properties, $\overline{\Sigma}_{\infty}$ must be ``regular at infinity'' in the sense that outside of a large compact set, it is the finite union of a graphs, all over the same fixed plane, of functions with nice asymptotic behavior, see \cite{Schoen:symmetry,Tysk:finite-index}. In particular, we may take $R>0$ so that $\overline{\Sigma}_{\infty}$ intersects $\partial B_{R}(0)$ transversely and
\[
|\sff_{\overline{\Sigma}_{\infty}}|(x) d_{\RR^{n}}(x,0) < \frac 1 4
\]
for $x \in\overline{\Sigma}_{\infty}\setminus B_{R}(0)$. 

We claim that for $j$ sufficiently large,
\begin{equation}\label{eq:curv-14-nD-1pt}
|\sff_{\Sigma_{j}}|(x)d_{g_{j}}(x,p_{j}) < \frac 1 4
\end{equation}
for $x \in \Sigma_{j} \cap \left( B_{2}(0) \setminus B_{R/\lambda_{j}}(p_{j})\right)$. If this holds, then Lemma \ref{lemm:ann-decomp} easily is seen to imply that for $j$ sufficiently large, all of the hypersurfaces $\Sigma_{j}\cap B_{1}(0)$ are diffeomorphic (here, we have used the fact that ``regular ends'' are diffeomorphic to $\SS^{n-2}\times (0,1)$ with the standard smooth structure).

On the other hand, if \eqref{eq:curv-14-nD-1pt} does not hold, we may choose $\delta_{j}$ to be the smallest radius\footnote{Observe that when $n=3$, the half-space theorem affords us considerably more flexibility in this argument. Here, we must tranfer ``connectedness'' from larger to smaller scales by choosing the largest scale where the estimate \eqref{eq:curv-14-nD-1pt} is violated.}   greater than $R/\lambda_{j}$ so that
\[
|\sff_{\Sigma_{j}}|(x)d_{g_{j}}(x,p_{j}) < \frac 1 4
\]
holds for $x \in \Sigma_{j} \cap \left( B_{2}(0) \setminus B_{\delta_{j}}(p_{j})\right)$. Note that for $j$ sufficiently large, such a $\delta_{j}$ exists and moreover $\delta_{j}\to 0$. This follows from fact that $\Sigma_{j}$ converges smoothly to $\{x^{n}=0\}$ away from $p_{\infty}$. 

Define
\[
\hat \Sigma_{j} := \delta_{j}^{-1} (\Sigma_{j}-p_{j}).
\]
Passing to a subsequence, there is $\hat\Sigma_{\infty} \subset \RR^{n}$ so that $\hat\Sigma_{j}$ converges locally smoothly with finite multiplicity to $\hat\Sigma_{\infty}$ away from $\{0\}$, and converges in the sense of varifolds in $B_{1}(0)$. Because $\hat\Sigma_{\infty}$ has finite index, we may apply Proposition \ref{prop:high-dim-remov-sing} to see that the singularity at $\{0\}$ is removable. In particular (after relabeling the hypersurface), $\hat\Sigma_{\infty}$ is an embedded minimal hypersurface in $\RR^{n}$ with $\Index(\hat\Sigma_{\infty}) \leq I$ and $\vol(\hat\Sigma_{\infty}\cap B_{r}(0)) \leq \Lambda r^{n-1}$. In particular, it is regular at infinity and has finitely many components. Hence, we may choose $\gamma \geq 1$ so that $\partial B_{\gamma}(0)$ intersects each component transversely, and $\hat\Sigma_{\infty}\cap\partial B_{\gamma}(0)$ is the disjoint union of finitely many manifolds diffeomorphic to $\SS^{n-2}$ with the standard smooth structure. 

By choice of $\delta_{j}$, the curvature estimates \eqref{eq:curv-14-nD-1pt} hold for $x \in  \Sigma_{j} \cap \left( B_{2}(0) \setminus B_{\gamma \delta_{j}}(p_{j})\right)$. Applying Lemma \ref{lemm:ann-decomp}, we see that $\Sigma_{j} \cap \left( B_{2}(0) \setminus B_{\gamma \delta_{j}}(p_{j})\right)$ is diffeomorphic to the union of annular regions. In particular, $\Sigma_{j}\cap B_{\gamma\delta_{j}}(p_{j})$ must be connected (because we have assumed that $\Sigma_{j}$ is connected in \hyperlink{defi:beth}{$(\beth)$}). From this, we see that $\hat\Sigma_{\infty}$ is connected. Observe that the convergence of $\hat\Sigma_{j}$ to $\hat\Sigma_{\infty}$ cannot be smooth at $\{0\}$ by choice of $R$ and the assumption that $\delta_{j} \geq R/\lambda_{j}$. In particular, the convergence of $\hat\Sigma_{j}$ to $\hat\Sigma_{\infty}$ must occur with multiplicity at least two, so $\hat\Sigma_{\infty}$ is (two-sided) stable and thus a plane; note that this uses the fact that $\hat\Sigma_{\infty}$ is connected.\footnote{If we did not arrange for $\hat\Sigma_{\infty}$ to be connected, then we could only conclude that it contained a plane through the origin but it could have other non-flat components.} The convergence of $\hat\Sigma_{j}$ to $\hat\Sigma_{\infty}$ occurs smoothly near $\partial B_{1}(0)$. This contradicts the choice of $\delta_{j}$ (namely that \eqref{eq:curv-14-nD-1pt} fails at some point in $\Sigma_{j}\cap \partial B_{\delta_{j}}(p_{j})$). This completes the proof in the case that $|\cB_{j}| = 1$.

Now, we consider the case of $|\cB_{\infty}| \geq 2$. Pick $\delta > 0$ so that $\cB_{\infty}$ is $4\delta$-separated. In particular, if $B$ is a component of $B_{\delta}(\cB_{\infty})$, then for $j$ sufficiently large, we see that
\[
\Index(\Sigma_{j}\cap B) \leq I-1.
\]
We may choose $\varepsilon_{j}\to 0$ sufficiently slowly so that $\cB_{j}\subset B_{\varepsilon_{j}/j}(\cB_{\infty})$,
\[
\liminf_{j\to\infty}\varepsilon_{j} \min_{p\in\cB_{j}} |\sff_{\Sigma_{j}}|(p) = \infty
\] 
and so that every connected component of $\Sigma_{j}\cap B_{\delta}(\cB_{\infty})$ intersects $B_{\varepsilon_{j}}(\cB_{\infty})$. That we can find $\varepsilon_{j}$ satisfying final condition is easily justified by combining the smooth convergence away from $\cB_{\infty}$ to $\{x^{n}=0\}$ with the varifold convergence. 

Consider $\Sigma_{j}'$ a connected component of $\Sigma_{j}\cap B_{\delta}(p_{\infty})$ for some $p_{\infty}\in\cB_{\infty}$. We claim that for $j$ sufficiently large,
\begin{equation}\label{eq:curv-14-nD-mult-pt}
|\sff_{\Sigma_{j}}|(x)d_{g_{j}}(x,p_{\infty}) < \frac 1 4
\end{equation}
for $ x \in \Sigma_{j}' \cap \left( B_{\delta}(p_{\infty})\setminus B_{\varepsilon_{j}}(p_{\infty}) \right)$. Suppose that we have proven \eqref{eq:curv-14-nD-mult-pt} for each component. By the monotonicity formula and the uniform volume bound in \hyperlink{defi:beth}{$(\beth)$}, there must be a bounded number of such components. Thus, by taking $j$ sufficiently large, we have that
\[
|\sff_{\Sigma_{j}}|(x)d_{g_{j}}(x,p_{\infty}) < \frac 1 4
\]
for $x \in \Sigma_{j}\cap \left( B_{\delta}(p_{\infty})\setminus  B_{\varepsilon_{j}}(p_{\infty})\right)$. The inductive step (it is not hard to see that it is applicable to each connected component of $\Sigma_{j}\cap B_{\varepsilon_{j}}(p_{\infty})$, by how we chose $\varepsilon_{j}$), along with Lemma \ref{lemm:ann-decomp} and these bounds easily show that after passing to a subsequence each hypersurface $\Sigma_{j}\cap B_{\delta}(p_{\infty})$ is diffeomorphic. Passing to a further subsequence, we may arrange that each hypersurface $\Sigma_{j}\cap B_{\delta}(\cB_{\infty})$ is diffeomorphic. Now, since $\Sigma_{j}\setminus B_{\delta/2}(\cB_{\infty})$ converges smoothly (with finite multiplicity) to $\{x^{n}=0\}\setminus B_{\delta/2}(\cB_{\infty})$, there are only a finite number of ways that the hypersurfaces $\Sigma_{j}\cap B_{\delta}(\cB_{\infty})$ could join up with $\Sigma_{j}\cap \left( B_{1}(0) \setminus B_{\delta}(\cB_{\infty})\right)$, which is diffeomorphic to a disjoint union of finitely many copies of the ``planar region'' $\{x^{n}=0\}\cap \left(B_{1}(0)\setminus B_{\delta/2}(\cB_{\infty}) \right)$. Hence, as usual it remains to prove \eqref{eq:curv-14-nD-mult-pt} for each connected component $\Sigma_{j}'$. 

The argument is similar to the one point of concentration above. If \eqref{eq:curv-14-nD-mult-pt} failed, then we could choose $\delta_{j}\geq \varepsilon_{j}$ to be the smallest number so that \eqref{eq:curv-14-nD-mult-pt} held for $x \in \Sigma_{j}' \cap \left( B_{\delta}(p_{\infty})\setminus B_{\delta_{j}}(p_{\infty})\right)$. As before, $\delta_{j}\to 0$. The surface
\[
\hat\Sigma_{j}' := \delta_{j}^{-1}(\Sigma_{j}'-p_{\infty})
\]
converges after passing to a subsequence to $\hat\Sigma_{\infty}'$. Now, we may argue exactly as in the one point case to choose $\gamma\geq 1$ so that each component of $\hat\Sigma_{\infty}'$ intersects $\partial B_{\gamma}(0)$ transversely in spheres. Lemma \ref{lemm:ann-decomp} implies that $\Sigma_{j}'\cap \left(B_{\delta}(p_{\infty})\setminus B_{\gamma\delta_{j}}(p_{\infty})\right)$ is the union of annular regions. This implies that $\hat\Sigma_{\infty}'$ is connected, and is thus a plane through the origin. This contradicts the choice of $\delta_{j}$ by the same argument as before. This completes the proof in the case that $|\cB_{\infty}| \geq 2$.

Finally, in the case that $|\cB_{\infty}| = 1$ and $|\cB_{j}| \geq 2$, we can rescale by the distance between the furthest two points of concentration. The proof proceeds just as in Proposition \ref{prop:mult-point-conc}, as long as we prove the crucial curvature estimates from the large to small scale, as we have done above. We omit the details. 
\end{proof}

Now, to finish the proof of Theorem \ref{theo:fin-top-type-Rn}, we first observe that it is not restrictive to assume that the hypersurfaces are connected (the volume bounds and monotonicity formula imply that there can be at most a bounded number of connected components). If $\Sigma_{j}$ was a sequence of pairwise non-diffeomorphic connected, embedded, minimal hypersurfaces in $\RR^{n}$ with $\vol(\Sigma\cap B_{R}(0)) \leq \Lambda R^{1-n}$ and $\Index(\Sigma) \leq I$, then because such surfaces are ``regular at infinity,'' we may rescale and rotate the $\Sigma_{j}$ so that outside of $B_{\tau_{0}/2}(0)$, the $\Sigma_{j}$ are graphical over $\{x^{n}=0\}$. This guarantees that in particular the $\Sigma_{j}\cap B_{1}(0)$ are pairwise non-diffeomorphic as well. It is not hard to show that $\Sigma_{j}\cap B_{r_{j}}(0)$ satisfies \hyperlink{defi:beth}{$(\beth)$}, so the proof follows from the previous proposition. 

The proof of Theorem \ref{theo:fin-top-type-Mn} also follows easily from the above proposition: for $\Sigma_{j}\subset (M^{n},g)$ as in the statement of Theorem \ref{theo:fin-top-type-Mn}, pairwise non-diffeomorphic, their curvature cannot be bounded. Combining the previous proposition with the usual Morse theory argument, we see that after passing to a subsequence, the $\Sigma_{j}$ are all diffeomorphic in small fixed balls containing the points of curvature blow-up. The other portion of $\Sigma_{j}$ converges smoothly, and there are only finitely many ways to connect the regions of large curvature to the regions of bounded curvature.

\appendix

\section{The genus of a surface}\label{app:genus-bdry}
\begin{defi}\label{defi:genus-non-orient}
For $\Sigma$ a non-orientable closed surface, we define the (non-orientable) genus of $\Sigma$ to be
\[
\genus(\Sigma) = \frac 12 \genus(\widehat\Sigma) 
\]
where $\widehat\Sigma$ is the oriented double cover. 
\end{defi}

\begin{defi}\label{defi:genus-bdry}
For a compact surface $\Sigma$ with boundary $\partial\Sigma$ consisting of one or more closed curves, we define $\genus(\Sigma)$ to be the genus of the closed surface formed by gluing disks to each boundary component. 
\end{defi}
Suppose that $\Sigma_{1},\Sigma_{2}$ are two oriented surfaces with boundary. If we form an oriented surface $\Sigma$ by gluing together $b$ boundary components, then from the well known formula $\chi(\Sigma) = \chi(\Sigma_{1}) + \chi(\Sigma_{2})$, we find that
\[
\genus(\Sigma) = \genus(\Sigma_{1})+\genus(\Sigma_{2}) + b-1.
\]
The reader should keep in mind the example of a torus thought of as a sphere with two disks removed, glued to an annulus (along the two boundary components); neither component has any genus in the sense of Definition \ref{defi:genus-bdry}, but obviously the torus is a genus one surface. 

As a consequence of this, we find
\begin{lemm}\label{lemm:adding-genus-ends}
Suppose that $\Sigma$ is a properly embedded surface in $B_{2}(0)\subset \RR^{3}$ so that there is a finite set of points $\cB\subset B_{1/2}(0)$ which are $3\varepsilon$-separated for some $\varepsilon \in (0,1/4)$ having the following properties:
\begin{enumerate}[itemsep=5pt, topsep=5pt]
\item The surface $\Sigma$ intersects $\partial B_{\varepsilon}(\cB)$ and $\partial B_{1}(0)$ transversely.
\item The surface $\Sigma\setminus B_{\varepsilon}(\cB)$ is topologically the union of finitely many components, each of which is topologically a disk with finitely many holes removed.
\item The surface $\Sigma\cap B_{\varepsilon}(\cB)$ is two-sided.
\item For each $p \in \cB$, we have an upper bound $r(p)$ on the genus of $\Sigma\cap B_{\varepsilon}(p)$ and an upper bound $m(p)$ on the number of boundary circles $\Sigma\cap \partial B_{\varepsilon}(p)$.
\end{enumerate}
Then, the genus of $\Sigma$ is bounded by 
\[
\genus(\Sigma) \leq \sum_{p\in\cB} \left(r(p)+m(p)-1\right)
\]
and the number of boundary circles of $\Sigma\cap B_{1}(0)$ is bounded by
\[
|\pi_{0}(\Sigma\cap \partial B_{1}(0))| \leq \sum_{p\in\cB} m(p).
\]
\end{lemm}

\section{Finite index surfaces in $\RR^{3}$}\label{app:finite-index-RR3}
The following theorem is a consequence of results due to Osserman \cite{Osserman:FTC} and Fischer-Colbrie \cite{Fischer-Colbrie:1985} in the two-sided case. The one sided case is due to Ros \cite{Ros:oneSided}.
\begin{theo}\label{theo:fin-index-imp-proper}
Suppose that $\Sigma\hookrightarrow \RR^{3}$ is a complete minimal injective immersion in $\RR^{3}$ with finite index. Then $\Sigma$ is two-sided, has finite total curvature, and is properly embedded. 
\end{theo}
\begin{proof}
By \cite[Theorem 2]{Fischer-Colbrie:1985} and \cite[Theorem 17]{Ros:oneSided}, finite index is equivalent to finite total curvature for a complete minimal immersion in $\RR^{3}$. Using \cite{Osserman:FTC}, we have that $\Sigma$ is conformally diffeomorphic to a punctured Riemann surface. Hence, so is the orientable double cover---this shows the orientable double cover has finite total curvature. By \cite{Schoen:2ends}, we find that $\Sigma$ is a proper embedding and is thus two-sided. 
\end{proof}

This, along with the half-space theorem for minimal surfaces of finite total curvature (which is a trivial consequence of \cite{Schoen:2ends}) implies. 

\begin{coro}\label{coro:limit-lam-struct}
Suppose that $\Lambda$ is a smooth lamination of $\RR^{3}$ with finite index. Then, it is either a single properly embedded surface of finite total curvature or else it consists only of parallel planes, i.e. after a rotation $\Lambda = \RR^{2}\times K$ for $K\subset \RR$ closed. 
\end{coro}

\section{Two-sidedness of embedded surfaces on small scales} \label{app:two-sided}

In this section, we record the following well known two-sidedness property of properly embedded surfaces. We include a short proof for completeness (cf.\ \cite{Samelson}).
\begin{lemm}\label{lem:two-sided-small-balls}
Suppose that $\Sigma^{n} \subset B_{1}(0)\subset \RR^{n+1}$ is a properly embedded hypersurface. Then, $\Sigma$ is two-sided. 
\end{lemm}
\begin{proof}
Suppose $\Sigma$ were one-sided. Then, we can find a loop $\gamma \subset B_{1}(0)$ so that $\gamma$ intersects $\Sigma$ in exactly one point. Because $B_{1}(0)$ is simply connected, $\gamma$ spans a disk $D$. Because $\Sigma$ is properly embedded, we can perturb $D$ away from its boundary so that $D$ is transverse to $\Sigma$. This is easily seen to be a contradiction.
\end{proof}

\section{Removable singularity results}\label{app:remov-sing}

The following result is well known, but we indicate the proof for completeness.
\begin{prop}[Properly embedded surfaces with curvature bounds]\label{prop:remov-sing-proper}
Suppose that $(M,g)$ is a complete Riemannian three-manifold and $p\in M$. Suppose that for $\varepsilon > 0$, $\Sigma \subset B_{\varepsilon}(p)\setminus\{p\}$ is a properly embedded minimal surface with
\[
|\sff_{\Sigma}|(x)d_{g}(x,p) \leq C.
\] 
Then, $\Sigma$ smoothly extends across $p$, i.e. there is $\widetilde\Sigma\subset B_{\varepsilon}(p)$ with $\Sigma = \widetilde\Sigma\setminus\{p\}$. 
\end{prop}
\begin{proof}
Because $\Sigma$ is proper it has finite area in $B_{\varepsilon}(p)\setminus \{p\}$. Hence, the monotonicity formula is applicable and we may consider a tangent cone to $\Sigma$ at $p$ (the tangent cone may not be unique). The assumed curvature estimate imples that any tangent cone is smooth away from $\{0\}$, so it is a single plane (possibly with multiplicity) through the origin, and blow-ups of $\Sigma$ converge smoothly away from $0$ to any such tangent cone. Combined with a blow-up argument, this shows that there is $\delta \in (0,\varepsilon)$ so that
\[
|\sff_{\Sigma}|(x) d_{g}(x,p) < \frac 1 4
\]
for all $x \in \Sigma \cap B_{\delta}(p)$. A Morse theory argument analogous to Lemma \ref{lemm:ann-decomp} implies that $\Sigma\cap B_{\delta}(p)$ is the union of topological planes and annuli. Hence, it has finite Euler characteristic. A properly embedded minimal surface with finite Euler characteristic is well known to extend across a point singularity, cf. \cite[Proposition 1]{ChSc85}.
\end{proof}

We will make use of the following Bernstein-type result due to Gulliver--Lawson \cite{GulliverLawson}; see also \cite[Lemma 3.3]{MePeRo13} and \cite[Lemma A.26]{CM:fixed-genus-5}.
\begin{theo}[Gulliver--Lawson's Bernstein theorem]\label{thm:GL-Bern}
Suppose that $\varphi :\Sigma \to \RR^{3}\setminus\{0\}$ is a non-empty two-sided stable minimal immersion which is complete away from $\{0\}$. Then the trace of $\varphi$ is a plane. 
\end{theo}

Using this, we show the following removable singularity result for two-sided stable laminations. The fundamental strategy is somewhat similar to \cite{MePeRo13}, but thanks to the stability hypothesis (which is considerably stronger than the assumptions in \cite{MePeRo13}), we are able to give a relatively short argument, inspired by ideas in \cite{CCE}.

\begin{prop}[Removable singularities for two-sided stable laminations]\label{prop:remov-sing-two-sided-stab-lam}
Suppose that $(M,g)$ is a complete Riemannian three-manifold and $p \in M$. Suppose that for some $\varepsilon >0$, $\cL\subset B_{\varepsilon}(p)\setminus \{p\}$ is a minimal lamination with the property that any leaf $L \subset \cL$ has stable universal cover. Then, there is a smooth lamination $\widetilde \cL\subset B_{\varepsilon}(p)$ so that $\cL = \widetilde\cL\setminus\{p\}$.
\end{prop}
\begin{proof}
Because the claim is purely local, at several points we will replace $\cL$ with its intersection with some smaller ball $B_{\varepsilon'}(p)$. For simplicity, we will not relabel the resulting immersion or ball. Furthermore, we will always work in a normal coordinate system around $p$ (where we can assume the metric to be sufficiently close to Euclidean, by taking $\varepsilon>0$ sufficiently small). 

Observe that, taking $\varepsilon$ smaller if necessary, by Schoen's curvature estimates \cite{Sch83}, there is $C>0$ so that $|\sff_{\cL}|(x)d_{g}(x,p) \leq C$ for all $x \in \cL$. Hence, for any $\rho_{j}\to \infty$, passing to a subsequence, the laminations $\rho_{j}(\cL_{j}-p)$ converge to a smooth lamination $\cL_{\infty}$ of $\RR^{3}\setminus\{0\}$ away from $\{0\}$. Moreover each leaf in $\cL_{\infty}$ has stable universal cover and is complete away from $\{0\}$. Hence by Theorem \ref{thm:GL-Bern}, after rotating, $\cL_{\infty} = (\RR^{2}\times K) \setminus\{0\}$ for some closed set $K\subset\RR$.

Thus, taking $\varepsilon>0$ sufficiently small, we may guarantee that 
\begin{equation}\label{eq:remov-sing-imp-curv}
|\sff_{\cL}|(x) d_{g}(x,p) < \frac 1 4
\end{equation}
for all $x \in \cL$. 

Now, pick any leaf $L\subset \cL$ so that $p$ is in the (topological) closure of $L$. We claim that for any $\rho_{j}\to 0$, passing to a subsequence and rotating the coordinate chart, the surfaces $L_{j} : = \rho_{j}L$ converge to the lamination $(\RR^{2}\times \{0\}) \setminus\{0\}\subset \RR^{3}\setminus\{0\}$. To prove this, by our above argument, it is sufficient to show that there is not another plane $\Pi = \RR^{2}\times \{z\}$ in the lamination limit of $\rho_{j}L$. If such a plane did exist, then by curvature estimates, we would have locally smooth convergence to $\Pi$ (because $0\not \in \Pi$). At the scale of $L$, this would imply that there was some $\delta \in (0,\varepsilon)$ sufficiently small, so that $L \cap B_{\delta}(p)\setminus\{p\}$ contains a properly embedded component $D$ diffeomorphic to a disk, intersecting $\partial B_{\delta}(p)$ transversely. 

Now, by the curvature estimates \eqref{eq:remov-sing-imp-curv}, we claim that a Morse theory argument along the lines of Lemma \ref{lemm:ann-decomp} shows that $L \setminus B_{\delta}(p)$ contains an annular region connecting the disk $D$ to $\partial B_{\varepsilon}(0)$. Given this, because $L$ is connected (by definition) we see that it must be the union of $D$ with this annular region. This implies that $p$ cannot be in the closure of $L$, a contradiction. We emphasize that $L$ is not assumed to be proper, so we cannot simply apply Lemma \ref{lemm:ann-decomp}. However, the given curvature estimates and resulting compactness properties of the leaves are sufficient to handle the lack of properness in the proof of the mountain pass lemma; this has been carried out in detail in \cite[Appendix E]{CCE} in a completely analogous situation.

Thus, we may take $\delta\in(0,\varepsilon/3)$ sufficiently small so that after rotating the normal coordinate system, $L\cap \left(B_{3\delta}(p)\setminus B_{\delta}(p)\right)$
intersects $\partial B_{2\delta}(p)$ transversely, is contained in a $\delta/10$ neighborhood of the coordinate plane $\RR^{2}\times \{0\}$, and is a multigraph over this plane (cf.\ \cite[Lemma 4.1]{CCE}). Hence, the intersection $L \cap \partial B_{2\delta}(p)$ is the union of simple closed curves and injectively immersed curves which ``spiral near the equator'' of the sphere $\partial B_{2\delta}(p)$. First, assume that $L\cap \partial B_{2\delta}(p)$ contains a simple closed curve. Then, by the curvature estimate \eqref{eq:remov-sing-imp-curv} and the Morse theory argument used above, we see that $L\cap \partial B_{3\delta}(p)$ is a properly embedded annulus in $B_{3\delta}(p)\setminus\{p\}$. Hence, $L$ extends across $\{p\}$ by Proposition \ref{prop:remov-sing-proper}. 

Thus, it remains to consider the case that $L\cap \partial B_{2\delta}(p)$ consists of one or more spiraling curves. Our argument here is analogous to the technique of passing to the top sheet in \cite[Proposition 4.2]{CCE}. We have seen that $L\cap \partial B_{2\delta}(p)$ is contained in an $\delta/10$ neighborhood of the equator $(\RR^{2}\times \{0\} \cap \partial B_{2\delta}(p)$.   Taking a sequence of points $w_{j} \in L\cap \partial B_{2\delta}(p)$ with $x^{3}(w_{j})$ approaching $\sup_{w \in L\cap B_{2\delta}(p)} x^{3}(w)$, after passing to a subsequence, the points $w_{j}$ converge to $w'$, which lies in a ``top sheet'' $L'\subset \cL$. By construction, $L'\cap \partial B_{2\delta}(p)$ will contain a simple closed curve, and thus the Morse theory argument guarantees that $L' \cap B_{3\delta}(p)$ is a properly embedded topological disk or annulus. By Proposition \ref{prop:remov-sing-proper}, $L'$ extends across $\{p\}$. Similarly, we may pass to the bottom sheet to find a properly embedded $L'' \subset \cL$ which extends across $\{p\}$. Note that by construction $L'\not = L''$ (otherwise, there could not be any spiraling). 

Note that by the maximum principle, it cannot happen that both $L'$ and $L''$ contain $p$ (because they are smoothly, properly embedded in $B_{3\delta}(p)$), so we assume that $p\not \in L'$. However, this leads to a contradiction as follows: by construction and the curvature estimates, we can find a sufficiently small tubular neighborhood $\cU$ of $L'\cap B_{3\delta}(p)$ so that $L \cap \cU$ is a multigraph over $L'\cap B_{3\delta}(p)$. Because $L'\cap B_{3\delta}(p)$ is a disk (and thus is simply connected), this shows that at least one component of $L\cap \cU$ must be a disk, contradicting the assumed spiraling behavior of $L$. 
\end{proof}

\begin{rema}
We remark that by combining the curvature estimates \cite{Sch83,Ros:oneSided} with the removable singularity result \cite{MePeRo13}, the following strengthened version of Proposition \ref{prop:remov-sing-two-sided-stab-lam} holds: suppose that $\cL \subset B_{\varepsilon}(p)\setminus\{p\}$ is a lamination so that every leaf has a cover that is stable. Then, $\cL$ extends smoothly across $\{p\}$. Note that that this version is compatible with one-sided stability of leaves, while Proposition \ref{prop:high-dim-remov-sing} requires two-sided stability.\footnote{We emphasize that one-sided stability does not necessarily imply that the universal cover is stable; consider $\RR P^{2}$ in $\RR P^{3}$.}

We have not been able to find a self contained proof of this strengthened assertion, due to the fact that we do not know if the one-sided version of Theorem \ref{thm:GL-Bern} holds. It would thus be interesting to know if there can be a non-flat, one-sided stable immersion in $\RR^{3}$ that is complete away from the origin. Note that the surface cannot be injectively immersed, because then by the curvature estimates \cite{Ros:oneSided}, one could take the lamination closure away from $\{0\}$ and then apply \cite{MePeRo13}. 

Fortunately, for our purposes, the removable singularity result contained in Proposition \ref{prop:remov-sing-two-sided-stab-lam} is sufficient.
\end{rema}

\begin{coro}[cf.\ \cite{MePeRo13}]\label{coro:lam-min-pts-smooth-or-plane}
Suppose that $\cB\subset \RR^{3}$ is a finite set of points and $\Lambda \subset \RR^{3}\setminus \cB$ is a smooth lamination of $\RR^{3}\setminus \cB$ so that
\[
|\sff_{\Lambda}|(x) d_{\RR^{3}}(x,\cB) \leq C.
\]
Then, either $\Lambda$ extends smoothly across the points $\cB$ or it contains a plane. 
\end{coro}
\begin{proof}
If each leaf is properly embedded in $\RR^{3}\setminus \cB$, then Proposition \ref{prop:remov-sing-proper} implies that $\Lambda$ extends across $\cB$. Otherwise, $\Lambda$ contains some limit leaf $L$, which has stable universal cover by standard arguments (cf.\ \cite[Lemma A.1]{MeRo06}). Then the closure of $L$ in $\Lambda$ is a non-empty lamination consisting entirely of leaves with stable universal cover. Proposition \ref{prop:remov-sing-two-sided-stab-lam} then guarantees that the closure of $L$ extends smoothly across $\cB$. Hence, we see that $L$ must be a plane, by \cite{Fischer-Colbrie-Schoen,doCarmoPeng,Pogorelov}. 
\end{proof}

Finally, we need the following removable singularity result valid in higher dimensions. 

\begin{prop}\label{prop:high-dim-remov-sing}
Suppose that $(M^{n},g)$ is a complete Riemannian $n$-dimensional manifold, for $4\leq n\leq 7$ and for some $\varepsilon>0$, $\Sigma\subset B_{\varepsilon}(p)\setminus\{p\}$ is a properly embedded stable minimal hypersurface. Then, $\Sigma$ smoothly extends across $p$.
\end{prop}
\begin{proof}
Stability of the hypersurface implies that $|\sff_{\Sigma}|(x) d_{g}(x,p) \leq C$. This follows from a blow-up argument as in Lemma \ref{lemm:curv-est-nD}, based on \cite{SSY,Schoen-Simon:1981} and the fact that a properly embedded hypersuface in $\RR^{n}$ is two-sided (cf.\ \cite{Samelson}).

As usual, the claim is purely local, so there is no harm with assuming that $\varepsilon> 0$ is sufficiently small. The curvature estimates and properness guarantee that $\vol_{g}(\Sigma) < \infty$. Thus, the monotonicity formula allows us to consider the tangent cones to (the varifold closure of) $\Sigma$ at $p$. The curvature estimates guarantee all of the tangent cones have smooth, compact, connected, cross section in $\SS^{n-1}$. It is well known that (by Alexander duality, cf.\ \cite{Samelson}) compact embedded hypersurfaces separate $\SS^{n-1}$, and are thus two-sided. Thus, the tangent cones themselves are two-sided. Because the rescalings of $\Sigma_{j}$ converge smoothly away from the origin to the tangent cones, we may see that the cones are stable precisely in the sense needed to apply \cite{Simons:cones}. This allows us to conclude that all tangent cones to $\Sigma$ at $p$ are hyperplanes (possibly with multiplicity).

In particular, taking $\varepsilon >0$ smaller if necessary, we may arrange that
\[
|\sff_{\Sigma}|(x) d_{g}(x,p) \leq \frac 14.
\]
Then, using a Morse theoretic argument along the lines of Lemma \ref{lemm:curv-est-nD} (and taking $\varepsilon>0$ smaller if necessary), we may arrange that each of the (bounded number of) components of $\Sigma$ is diffeomorphic to $\SS^{n-2}\times (0,1)$ and each component intersects $\partial B_{s}(p)$ transversely in a connected submanifold for $s \leq \varepsilon$. From this, it is not hard to see that any tangent cone to (the varifold closure of) the hypersurface $\widetilde{\Sigma}$ at $p$ is a multiplicity one plane. Thus, by Allard's theorem \cite{Allard:varifold}, $\widetilde{\Sigma}$ extends smoothly across $p$. Using the maximum principle, we thus see that there can be only one component of $\Sigma$ whose closure includes $p$. This completes the proof. 
\end{proof}

\section{Examples of degeneration}\label{app:exam-degen}

We give examples to illustrate the ``one point of concentration'' and ``multiple points of concentration'' discussed in Propositions \ref{prop:one-point-conc} and \ref{prop:mult-point-conc}. Recall that (see \cite{Costa:1984,HoffmanMeeks,HoffmanKarcher}) the Costa surface $\Sigma^{(1)}\subset \RR^{3}$ is an embedded minimal surface with genus one and three ends (one flat end, and two catenoidal ends) Furthermore, Hoffman--Meeks have shown that it is possible to deform the flat end of $\Sigma_{1}$ into a catenoidal end, producing a family $\Sigma^{(t)}$ for $t\geq 1$ of embedded genus three embedded minimal surfaces with three catenoidal ends. As $t\to\infty$, the logarithmic growth of middle catenoidal end approaches that of the  other end pointing in the same direction. See \cite[Figure 3.2]{HoffmanKarcher} or Figure \ref{fig:Costa-deform} for an illustration of the deformation family for $t$ large.

The exact index of $\Sigma^{(t)}$ seems to be unknown for $t>1$ (note that $\Index(\Sigma^{(1)}) = 5$ by \cite{Nayatani:CHM}). However, because the family $\Sigma^{(t)}$ has uniformly bounded total curvature, the main result of \cite{Tysk} implies that $\Index(\Sigma^{(t)}) \leq I$ for some $I \in \NN$. We will always assume that $\Sigma^{t}$ is scaled so that the second fundamental form's maximal norm is equal to $1$ and so that the line $\{x^{1}=x^{2}=0\}$ is the intersection of the two planes of reflection symmetry. 

\begin{figure}[h]
\begin{tikzpicture}

\begin{scope}[scale = 2]
\clip (-3,-1) rectangle (3,1);
\draw (-5,0) to [bend right = 1] (5,0);
\draw (-5,.7) to [bend right = 2] (5,.7);
\draw (-5,-.839) to [bend left = 3] (5,-.839);

\filldraw [white] (-.2,-.3) rectangle (.2,1);
\draw (-.21,-.05) to [bend right=90, looseness = .7] (-.21,.6);
\draw (.21,-.05) to [bend left=90, looseness = .7] (.21,.6);
\draw (-.08,.28) to [bend right = 30] (.08,.28);
\draw [opacity = .4] (-.08,.28) to [bend left = 30] (.08,.28);

\foreach \x in {-1.09, 1.09}{
\begin{scope}[shift = {(\x,-.645)}]
	\filldraw [white] (-.2,-.3) rectangle (.2,1);
	\draw (-.21,-.05) to [bend right=90, looseness = .7] (-.21,.6);
	\draw (.21,-.05) to [bend left=90, looseness = .7] (.21,.6);
	\draw (-.08,.28) to [bend right = 30] (.08,.28);
	\draw [opacity = .4] (-.08,.28) to [bend left = 30] (.08,.28);
\end{scope}
}
\end{scope}
\filldraw (0,.52) circle (1pt);
\draw [dashed] (0,.52) circle (.8);

\filldraw (-2.17,-.785) circle (1pt);
\draw [dashed] (-2.17,-.785) circle (.8);

\filldraw (2.18,-.77) circle (1pt);
\draw [dashed] (2.18,-.77) circle (.8);
\end{tikzpicture}
\caption{This depicts the Hoffman--Meeks deformation $\Sigma^{(t)}$ of the Costa surface for $t$ large. The surface looks like three planes joined by three catenoidal necks. The marked points are the points of curvature concentration, and at the scale of curvature around any of them, $\Sigma^{(t)}$ is geometrically close to a catenoid. By choosing the scaling appropriately, one may arrange that the point of concentration stay at a bounded distance, or all converge to the origin.}\label{fig:Costa-deform}
\end{figure}
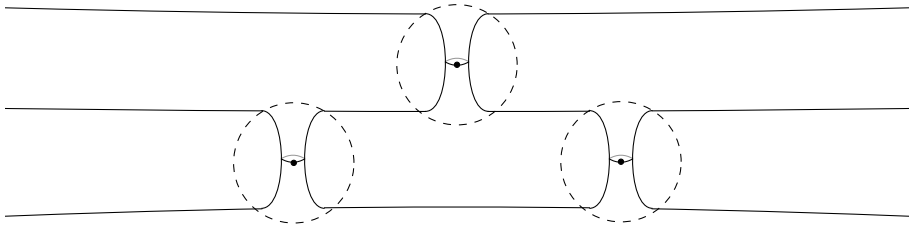

First, to illustrate the case of one point of concentration simply consider 
\[
\Sigma_{j} = \left(\frac 1 j \Sigma^{(1)}\right) \cap B_{r_{j}}(0)
\] 
for some $r_{j}\to\infty$. As $r\to\infty$, this converges smoothly away from the origin to the plane $\RR^{2}\times\{0\}$ with multiplicity $3$. Clearly this satisfies \hyperlink{defi:aleph}{$(\aleph)$} with one point of curvature concentration $p_{j}$ converging to $0$ as $j\to\infty$. Rescaling the sequence at the scale of curvature around $p_{j}$ simply finds $\Sigma_{1}$. 

Second, to illustrate the various possibilities for multiple points of concentration, we must describe the behavior of $\Sigma^{(t)}$ as $t\to\infty$ more precisely. One may show that for $\rho_{j}\to 0$ sufficiently quickly,
\[
\Sigma_{j} : = \left( \rho_{j}\Sigma^{(j)}\right) \cap B_{r_{j}}(0)
\]
looks like three nearby disks, with the middle disk jointed to the bottom disk by two catenoidal necks in equal and opposite directions from the origin and the middle disk joined to the top disk by a single catenoidal neck near the origin. This is well illustrated in \cite[Figure 3.2]{HoffmanKarcher}; see also Figure \ref{fig:Costa-deform}. To establish this picture rigorously, one may appeal to \cite[Theorem 2]{Ros:compactnessFTC} and the fact that the catenoid is the only non-flat embedded minimal surface $\hat\Sigma \subset \RR^{3}$ with total curvature $\int_{\hat\Sigma} \kappa >  - 12 \pi$ (cf.\ \cite[Theorem 3.1]{HoffmanKarcher}). 

In particular, the blow-up set $\cB_{j}$ has $|\cB_{j}| = 3$ and rescaling around any such point is a catenoid. However, if we chose $\rho_{j}\to 0$ sufficiently quickly, $\cB_{j}$ converges to $\cB_{\infty}= \{0\}$ as $j\to\infty$. On the other hand, if $\rho_{j}\to 0$ at precisely the correct rate, it is clear that (after a rotation) $\cB_{j}$ converges to $\{(0,0,0), (\pm 1,0,0)\} \subset \RR^{3}$. This example, and considerably more refined examples are discussed in great detail in \cite{Traizet}. 

\bibliography{bib} 

\providecommand{\bysame}{\leavevmode\hbox to3em{\hrulefill}\thinspace}
\providecommand{\MR}{\relax\ifhmode\unskip\space\fi MR }
\providecommand{\MRhref}[2]{%
  \href{http://www.ams.org/mathscinet-getitem?mr=#1}{#2}
}
\providecommand{\href}[2]{#2}
\begin{thebibliography}{MPR16}

\bibitem[All72]{Allard:varifold}
William~K. Allard, \emph{On the first variation of a varifold}, Ann. of Math.
  (2) \textbf{95} (1972), 417--491. \MR{0307015 (46 \#6136)}

\bibitem[Bog08]{Bogopolski}
Oleg Bogopolski, \emph{Introduction to group theory}, EMS Textbooks in
  Mathematics, European Mathematical Society (EMS), Z\"urich, 2008, Translated,
  revised and expanded from the 2002 Russian original. \MR{2396717}

\bibitem[BS15]{BuzanoSharp:topHighDim}
Reto Buzano and Ben Sharp, \emph{Qualitative and quantitative estimates for
  minimal hypersurfaces with bounded index and area}, To appear in Trans. Amer.
  Math. Soc. , available at \url{http://arxiv.org/abs/1512.01047} (2015).

\bibitem[BW69]{BrWo69}
Glen~E. Bredon and John~W. Wood, \emph{Non-orientable surfaces in orientable
  {$3$}-manifolds}, Invent. Math. \textbf{7} (1969), 83--110. \MR{0246312 (39
  \#7616)}

\bibitem[Car15]{Carlotto:generic}
Alessandro Carlotto, \emph{Generic finiteness of minimal surfaces with bounded
  {M}orse index}, to appear in Ann. Scuola Norm. Sup. Pisa, available at
  \url{http://arxiv.org/abs/1509.07101} (2015).

\bibitem[CCE16]{CCE}
Alessandro Carlotto, Otis Chodosh, and Michael Eichmair, \emph{Effective
  versions of the positive mass theorem}, Invent. Math. \textbf{206} (2016),
  no.~3, 975--1016. \MR{3573977}

\bibitem[CDL05]{CD}
Tobias~H. Colding and Camillo De~Lellis, \emph{Singular limit laminations,
  {M}orse index, and positive scalar curvature}, Topology \textbf{44} (2005),
  no.~1, 25--45. \MR{2103999 (2005h:53102)}

\bibitem[CG14]{ColdingGabai}
Tobias~H. Colding and David Gabai, \emph{Effective finiteness of irreducible
  {H}eegaard splittings of non {H}aken 3-manifolds}, preprint, available at
  \url{http://arxiv.org/abs/1411.2509} (2014).

\bibitem[CH06]{CoHi06}
Tobias~H. Colding and Nancy Hingston, \emph{Geodesic laminations with closed
  ends on surfaces and {M}orse index; {K}upka-{S}male metrics}, Comment. Math.
  Helv. \textbf{81} (2006), no.~3, 495--522. \MR{2250851 (2007f:53037)}

\bibitem[CM00]{ColdingMinicozzi:no-area-bds}
Tobias~H. Colding and William~P. Minicozzi, II, \emph{Embedded minimal surfaces
  without area bounds in 3-manifolds}, Geometry and topology: {A}arhus (1998),
  Contemp. Math., vol. 258, Amer. Math. Soc., Providence, RI, 2000,
  pp.~107--120. \MR{1778099 (2001i:53012)}

\bibitem[CM15]{CM:fixed-genus-5}
\bysame, \emph{The space of embedded minimal surfaces of fixed genus in a
  3-manifold {V}; fixed genus}, Ann. of Math. (2) \textbf{181} (2015), no.~1,
  1--153. \MR{3272923}

\bibitem[CM16]{ChMa14}
Otis Chodosh and Davi Maximo, \emph{On the topology and index of minimal
  surfaces}, J. Differential Geom. \textbf{104} (2016), no.~3, 399--418.
  \MR{3568626}

\bibitem[Cos84]{Costa:1984}
Celso~J. Costa, \emph{Example of a complete minimal immersion in {${\bf R}^3$}
  of genus one and three embedded ends}, Bol. Soc. Brasil. Mat. \textbf{15}
  (1984), no.~1-2, 47--54. \MR{794728 (87c:53111)}

\bibitem[CS85]{ChSc85}
Hyeong~In Choi and Richard Schoen, \emph{The space of minimal embeddings of a
  surface into a three-dimensional manifold of positive {R}icci curvature},
  Invent. Math. \textbf{81} (1985), no.~3, 387--394. \MR{807063 (87a:58040)}

\bibitem[CW83]{ChoiWang}
Hyeong~In Choi and Ai~Nung Wang, \emph{A first eigenvalue estimate for minimal
  hypersurfaces}, J. Differential Geom. \textbf{18} (1983), no.~3, 559--562.
  \MR{723817 (85d:53028)}

\bibitem[dCP79]{doCarmoPeng}
Manfredo do~Carmo and Chiakuei Peng, \emph{Stable complete minimal surfaces in
  {${\bf R}^{3}$} are planes}, Bull. Amer. Math. Soc. (N.S.) \textbf{1} (1979),
  no.~6, 903--906. \MR{546314 (80j:53012)}

\bibitem[Dea03]{Dea03}
Brian Dean, \emph{Compact embedded minimal surfaces of positive genus without
  area bounds}, Geom. Dedicata \textbf{102} (2003), 45--52. \MR{2026836
  (2005d:53011)}

\bibitem[EM08]{EjiriMicallef}
Norio Ejiri and Mario Micallef, \emph{Comparison between second variation of
  area and second variation of energy of a minimal surface}, Adv. Calc. Var.
  \textbf{1} (2008), no.~3, 223--239. \MR{2458236 (2009j:58019)}

\bibitem[FC85]{Fischer-Colbrie:1985}
Doris Fischer-Colbrie, \emph{On complete minimal surfaces with finite {M}orse
  index in three-manifolds}, Invent. Math. \textbf{82} (1985), no.~1, 121--132.
  \MR{808112 (87b:53090)}

\bibitem[FCS80]{Fischer-Colbrie-Schoen}
Doris Fischer-Colbrie and Richard Schoen, \emph{The structure of complete
  stable minimal surfaces in {$3$}-manifolds of nonnegative scalar curvature},
  Comm. Pure Appl. Math. \textbf{33} (1980), no.~2, 199--211. \MR{562550
  (81i:53044)}

\bibitem[FHS83]{FHS83}
Michael Freedman, Joel Hass, and Peter Scott, \emph{Least area incompressible
  surfaces in {$3$}-manifolds}, Invent. Math. \textbf{71} (1983), no.~3,
  609--642. \MR{695910 (85e:57012)}

\bibitem[GL86]{GulliverLawson}
Robert Gulliver and H.~Blaine Lawson, Jr., \emph{The structure of stable
  minimal hypersurfaces near a singularity}, Geometric measure theory and the
  calculus of variations ({A}rcata, {C}alif., 1984), Proc. Sympos. Pure Math.,
  vol.~44, Amer. Math. Soc., Providence, RI, 1986, pp.~213--237. \MR{840275
  (87g:53091)}

\bibitem[HI01]{Huisken-Ilmanen:2001}
Gerhard Huisken and Tom Ilmanen, \emph{The inverse mean curvature flow and the
  {R}iemannian {P}enrose inequality}, J. Differential Geom. \textbf{59} (2001),
  no.~3, 353--437. \MR{1916951 (2003h:53091)}

\bibitem[HK97]{HoffmanKarcher}
David Hoffman and Hermann Karcher, \emph{Complete embedded minimal surfaces of
  finite total curvature}, Geometry, {V}, Encyclopaedia Math. Sci., vol.~90,
  Springer, Berlin, 1997, pp.~5--93. \MR{1490038 (98m:53012)}

\bibitem[HM85]{HoffmanMeeks}
D.~Hoffman and W.~H. Meeks, III, \emph{A complete embedded minimal surface in
  {${\bf R}^3$} with genus one and three ends}, J. Differential Geom.
  \textbf{21} (1985), no.~1, 109--127. \MR{806705 (87d:53008)}

\bibitem[Jac70]{Jac70}
William Jaco, \emph{Surfaces embedded in {$M^{2}\times S^{1}$}}, Canad. J.
  Math. \textbf{22} (1970), 553--568. \MR{0267596 (42 \#2498)}

\bibitem[Kra09]{Kra09}
Joel Kramer, \emph{Embedded minimal spheres in 3-manifolds}, Ph.D. thesis, The
  Johns Hopkins University, 2009.

\bibitem[Li16]{Li:index-high-dim}
Chao Li, \emph{Index and topology of minimal hypersurfaces in {R}\string^n},
  preprint, available at \url{https://arxiv.org/abs/1605.09693} (2016).

\bibitem[LZ16]{LiZhou}
Haozhao Li and Xin Zhou, \emph{Existence of minimal surfaces of arbitrarily
  large {M}orse index}, Calc. Var. Partial Differential Equations \textbf{55}
  (2016), no.~3, Art. 64, 12. \MR{3509038}

\bibitem[MN13]{MaNe13}
Fernando~C. Marques and Andr\'e Neves, \emph{Existence of infinitely many
  minimal hypersurfaces in positive {R}icci curvature}, to appear in Invent.
  Math., available at \url{http://arxiv.org/abs/1311.6501} (2013).

\bibitem[MPR16]{MePeRo13}
William~H. Meeks, III, Joaqu{\'{\i}}n P{\'e}rez, and Antonio Ros, \emph{Local
  removable singularity theorems for minimal laminations}, J. Differential
  Geom. \textbf{103} (2016), no.~2, 319--362. \MR{3504952}

\bibitem[MR05]{MeRo05}
William~H. Meeks, III and Harold Rosenberg, \emph{The uniqueness of the
  helicoid}, Ann. of Math. (2) \textbf{161} (2005), no.~2, 727--758.
  \MR{2153399 (2006f:53012)}

\bibitem[MR06]{MeRo06}
\bysame, \emph{The minimal lamination closure theorem}, Duke Math. J.
  \textbf{133} (2006), no.~3, 467--497. \MR{2228460 (2007e:53007)}

\bibitem[Nay93]{Nayatani:CHM}
Shin Nayatani, \emph{Morse index and {G}auss maps of complete minimal surfaces
  in {E}uclidean {$3$}-space}, Comment. Math. Helv. \textbf{68} (1993), no.~4,
  511--537. \MR{1241471 (95b:58039)}

\bibitem[Oss64]{Osserman:FTC}
Robert Osserman, \emph{Global properties of minimal surfaces in {$E^{3}$} and
  {$E^{n}$}}, Ann. of Math. (2) \textbf{80} (1964), 340--364. \MR{0179701 (31
  \#3946)}

\bibitem[Pit76]{Pit76}
Jon~T. Pitts, \emph{Existence and regularity of minimal surfaces on
  {R}iemannian manifolds}, Bull. Amer. Math. Soc. \textbf{82} (1976), no.~3,
  503--504.

\bibitem[Pog81]{Pogorelov}
Aleksei~V. Pogorelov, \emph{On the stability of minimal surfaces}, Dokl. Akad.
  Nauk SSSR \textbf{260} (1981), no.~2, 293--295. \MR{630142 (83b:49043)}

\bibitem[Ros95]{Ros:compactnessFTC}
Antonio Ros, \emph{Compactness of spaces of properly embedded minimal surfaces
  with finite total curvature}, Indiana Univ. Math. J. \textbf{44} (1995),
  no.~1, 139--152. \MR{1336435 (96k:53011)}

\bibitem[Ros06]{Ros:oneSided}
\bysame, \emph{One-sided complete stable minimal surfaces}, J. Differential
  Geom. \textbf{74} (2006), no.~1, 69--92. \MR{2260928 (2007g:53008)}

\bibitem[Sam69]{Samelson}
Hans Samelson, \emph{Orientability of hypersurfaces in {$R^{n}$}}, Proc. Amer.
  Math. Soc. \textbf{22} (1969), 301--302. \MR{0245026 (39 \#6339)}

\bibitem[Sch83a]{Sch83}
Richard Schoen, \emph{Estimates for stable minimal surfaces in
  three-dimensional manifolds}, Seminar on minimal submanifolds, Ann. of Math.
  Stud., vol. 103, Princeton Univ. Press, Princeton, NJ, 1983, pp.~111--126.
  \MR{795231 (86j:53094)}

\bibitem[Sch83b]{Schoen:symmetry}
Richard~M. Schoen, \emph{Uniqueness, symmetry, and embeddedness of minimal
  surfaces}, J. Differential Geom. \textbf{18} (1983), no.~4, 791--809 (1984).
  \MR{730928 (85f:53011)}

\bibitem[Sch83c]{Schoen:2ends}
\bysame, \emph{Uniqueness, symmetry, and embeddedness of minimal surfaces}, J.
  Differential Geom. \textbf{18} (1983), no.~4, 791--809 (1984). \MR{730928
  (85f:53011)}

\bibitem[Sha15]{Sharp}
Ben Sharp, \emph{Compactness of minimal hypersurfaces with bounded index}, to
  appear in J. Diff. Geom., available at \url{http://arxiv.org/abs/1501.02703}
  (2015).

\bibitem[Sim68]{Simons:cones}
James Simons, \emph{Minimal varieties in riemannian manifolds}, Ann. of Math.
  (2) \textbf{88} (1968), 62--105. \MR{0233295 (38 \#1617)}

\bibitem[SS81]{Schoen-Simon:1981}
Richard Schoen and Leon Simon, \emph{Regularity of stable minimal
  hypersurfaces}, Comm. Pure Appl. Math. \textbf{34} (1981), no.~6, 741--797.
  \MR{634285 (82k:49054)}

\bibitem[SSY75]{SSY}
Richard Schoen, Leon Simon, and Shing-Tung Yau, \emph{Curvature estimates for
  minimal hypersurfaces}, Acta Math. \textbf{134} (1975), no.~3-4, 275--288.
  \MR{0423263 (54 \#11243)}

\bibitem[SY79]{ScYa79}
R.~Schoen and Shing~Tung Yau, \emph{Existence of incompressible minimal
  surfaces and the topology of three-dimensional manifolds with nonnegative
  scalar curvature}, Ann. of Math. (2) \textbf{110} (1979), no.~1, 127--142.
  \MR{541332 (81k:58029)}

\bibitem[SY82]{ScYa82}
Richard Schoen and Shing~Tung Yau, \emph{Complete three-dimensional manifolds
  with positive {R}icci curvature and scalar curvature}, Seminar on
  {D}ifferential {G}eometry, Ann. of Math. Stud., vol. 102, Princeton Univ.
  Press, Princeton, N.J., 1982, pp.~209--228. \MR{645740 (83k:53060)}

\bibitem[SY83]{ScYa83}
Richard Schoen and S.~T. Yau, \emph{The existence of a black hole due to
  condensation of matter}, Comm. Math. Phys. \textbf{90} (1983), no.~4,
  575--579. \MR{719436 (84k:83005)}

\bibitem[Tra02]{Traizet}
Martin Traizet, \emph{An embedded minimal surface with no symmetries}, J.
  Differential Geom. \textbf{60} (2002), no.~1, 103--153. \MR{1924593
  (2004c:53008)}

\bibitem[Tra04]{Traizet:balancing}
\bysame, \emph{A balancing condition for weak limits of families of minimal
  surfaces}, Comment. Math. Helv. \textbf{79} (2004), no.~4, 798--825.
  \MR{2099123 (2005g:53017)}

\bibitem[Tys87]{Tysk}
Johan Tysk, \emph{Eigenvalue estimates with applications to minimal surfaces},
  Pacific J. Math. \textbf{128} (1987), no.~2, 361--366. \MR{888524
  (88i:53102)}

\bibitem[Tys89]{Tysk:finite-index}
\bysame, \emph{Finiteness of index and total scalar curvature for minimal
  hypersurfaces}, Proc. Amer. Math. Soc. \textbf{105} (1989), no.~2, 429--435.
  \MR{946639 (89g:53093)}

\bibitem[Whi87]{Whi87}
B.~White, \emph{Curvature estimates and compactness theorems in {$3$}-manifolds
  for surfaces that are stationary for parametric elliptic functionals},
  Invent. Math. \textbf{88} (1987), no.~2, 243--256. \MR{880951 (88g:58037)}

\bibitem[Whi91]{White:bumpy}
Brian White, \emph{The space of minimal submanifolds for varying {R}iemannian
  metrics}, Indiana Univ. Math. J. \textbf{40} (1991), no.~1, 161--200.
  \MR{1101226 (92i:58028)}

\bibitem[Whi05]{White:brakke}
\bysame, \emph{A local regularity theorem for mean curvature flow}, Ann. of
  Math. (2) \textbf{161} (2005), no.~3, 1487--1519. \MR{2180405}

\bibitem[Whi15]{White:bumpy2}
Brian White, \emph{On the bumpy metrics theorem for minimal submanifolds},
  preprint, available at \url{http://arxiv.org/abs/1503.01803} (2015).

\bibitem[YY80]{YangYau}
Paul~C. Yang and Shing~Tung Yau, \emph{Eigenvalues of the {L}aplacian of
  compact {R}iemann surfaces and minimal submanifolds}, Ann. Scuola Norm. Sup.
  Pisa Cl. Sci. (4) \textbf{7} (1980), no.~1, 55--63. \MR{577325 (81m:58084)}

\end{thebibliography}
\bibliographystyle{amsalpha}

\end{document}